\documentclass[a4paper,11pt]{article}
\usepackage{amsmath,amsthm}
\usepackage{authblk}
\usepackage{aligned-overset}
\usepackage[style=numeric-comp,maxbibnames=99,bibencoding=utf8,firstinits=true]{biblatex}
\usepackage{booktabs}
\usepackage{cancel}
\usepackage[font=small]{caption}
\usepackage[hmargin={26mm,26mm},vmargin={30mm,35mm}]{geometry}
\usepackage{nicefrac}
\usepackage{paralist}
\usepackage[colorlinks, allcolors=blue]{hyperref}
\usepackage{subcaption}
\usepackage[compact,small]{titlesec}
\usepackage{tikz}
\usepackage{tikz-cd}
\usepackage{pgfplotstable}
\usepackage{newtxmath}
\usepackage{newtxtext}
\usepackage{enumitem}
\usepackage{bm}
\usepackage{algpseudocode}
\usepackage{algorithm}
\usepackage{thmtools}
\usepackage{thm-restate}

\usepackage{graphicx}
\graphicspath{{./pics/}}

\bibliography{pec-poincare}

\newcommand{\email}[1]{\href{mailto:#1}{#1}}

\allowdisplaybreaks

\newtheorem{theorem}{Theorem}
\newtheorem{proposition}[theorem]{Proposition}
\newtheorem{lemma}[theorem]{Lemma}
\newtheorem{corollary}[theorem]{Corollary}
\theoremstyle{remark}
\newtheorem{remark}[theorem]{Remark}
\theoremstyle{definition}

\newtheorem{example}[theorem]{Example}

\newcommand{\ul}[1]{\underline{{#1}}}

\newcommand{\Real}{\mathbb{R}}
\newcommand{\Natural}{\mathbb{N}}

\newcommand{\st}{\,:\,}

\DeclareMathOperator{\tr}{tr}

\DeclareMathOperator{\card}{card}
\DeclareMathOperator{\SPAN}{span}

\DeclareMathOperator{\Ker}{Ker}
\DeclareMathOperator{\Image}{Im}

\newcommand{\DIFF}{\mathrm{d}}
\newcommand{\KOSZUL}{\kappa}

\newcommand{\trimmed}{{-}}

\newcommand{\PL}[2]{\mathcal{P}_{#1}\Lambda^{#2}}

\newcommand{\PLtrim}[2]{\mathcal{P}_{#1}^{\trimmed}\Lambda^{#2}}

\newcommand{\ltproj}[2]{\pi^{\trimmed,#1}_{r,#2}}

\newcommand{\Mh}{\mathcal{M}_h}
\newcommand{\Ih}{\mathcal{S}_h}
\newcommand{\FM}[1]{\Delta_{#1}}

\newcommand{\pf}{{\partial f}}
\newcommand{\dtop}{n}

\newcommand{\ffp}{f\!f'}
\newcommand{\orffp}{\epsilon_{\ffp}}

\newcommand{\uH}[2]{\underline{X}_{r,#2}^{#1}}

\newcommand{\uI}[2]{\underline{I}_{r,#2}^{#1}}

\newcommand{\vvvert}{\vert\kern-0.25ex\vert\kern-0.25ex\vert}
\newcommand{\norm}[2]{\Vert #2\Vert_{#1}}
\newcommand{\Norm}[2]{\left\Vert #2\right\Vert_{#1}}
\newcommand{\opn}[2]{\vvvert #2\vvvert_{#1}}

\newcommand{\mc}{\mathcal}
\newcommand\chain[1]{C_{#1}}
\newcommand\cochain[1]{C^{#1}}
\newcommand\bd[1]{\partial_{#1}}
\newcommand\cobd[1]{\delta^{#1}}
\newcommand\cycle[1]{Z_{#1}}
\newcommand{\op}[1]{I^{#1}}
\newcommand{\iop}[1]{J^{#1}}
\newcommand{\cspace}[1]{Z^\mathsf{c}_{#1}}

\let\boundary\relax
\newcommand\boundary[1]{B_{#1}}

\newcommand\coboundary[1]{B^{#1}}
\newcommand{\inner}[2]{\langle {#1} \, , {#2} \rangle}

\newcommand{\mat}{\mathbb}

\DeclareMathOperator{\supp}{supp}
\newcommand{\Pa}[2]{P_{#1}(#2)} %partition of canonical basis
\newcommand{\Alt}{\mathrm{Alt}}

\newcommand{\multi}[3][A]{\arrow[#2, phantom, "#3"{name=#1, inner sep=1ex}]}
\begin{document}

\title{Uniform Poincar\'{e} inequalities for the discrete de Rham complex of differential forms}
\author[1]{Daniele A. Di Pietro}
\author[1,3]{J\'{e}r\^{o}me Droniou}
\author[2]{Marien-Lorenzo Hanot}
\author[1]{Silvano Pitassi}
\affil[1]{IMAG, Univ Montpellier, CNRS, Montpellier, France\\
  \email{daniele.di-pietro@umontpellier.fr}, %
  \email{jerome.droniou@umontpellier.fr}
  \email{silvano.pitassi@umontpellier.fr}%
}
\affil[2]{Univ. Lille, UMR 8524 - Laboratoire Paul Painlevé, CNRS, Inria, France\\
  \email{marien-lorenzo.hanot@univ-lille.fr}
}
\affil[3]{School of Mathematics, Monash University, Melbourne, Australia.
}

\maketitle

\begin{abstract}
  In this paper we prove discrete Poincaré inequalities that are uniform in the mesh size for the discrete de Rham complex of differential forms developed in [Bonaldi, Di Pietro, Droniou, and Hu, \emph{An exterior calculus framework for polytopal methods}, J. Eur. Math. Soc., to appear, arXiv preprint \href{http://arxiv.org/abs/2303.11093}{2303.11093 [math.NA]}]. We unify the underlying ideas behind the Poincaré inequalities for all differential operators in the sequence, extending the known inequalities for the gradient, curl, and divergence in three-dimensions to polytopal domains of arbitrary dimension and general topology. A key step in the proof involves deriving specific Poincaré inequalities for the cochain complex supported on the polytopal mesh. These inequalities are of independent interest, as they are useful, for instance, in establishing the existence and stability, on domains of generic topology, of solutions of schemes based on Mimetic Finite Differences, Compatible Discrete Operators or Discrete Geometric Approach.
\end{abstract}

%% \tableofcontents

%------------------------------------------------------------------------------%

\section{Introduction}

The well-posedness of important classes of partial differential equations can be related to the cohomology of the de Rham complex through Poincar\'{e}-type inequalities.
Such inequalities state that the $L^2$-norm of forms that lie in an orthogonal complement of the kernel of the exterior derivative is controlled by the $L^2$-norm of the latter.
Mimicking Poincar\'{e} inequalities on the discrete level is required for the stability of \emph{compatible} numerical schemes, the design of which hinges on discrete versions of the de Rham complex.

In the context of conforming Finite Element de Rham complexes, Poincaré inequalities can be derived through bounded cochain projections; see \cite[Chapter~5]{Arnold:18}, \cite{Arnold.Falk.ea:06}, and also \cite{Christiansen.Licht:20} for a recent generalization.
Owing to the need to identify an underlying computable space of (polynomial) functions, Finite Element constructions are, however, typically limited to conforming meshes with elements of simple shape.
Recent works have pointed out the possibility of constructing arbitrary-order discrete de Rham complexes on general polytopal meshes, either through the use of virtual spaces, standard differential operators, and projections \cite{Beirao-da-Veiga.Brezzi.ea:16,Beirao-da-Veiga.Brezzi.ea:18}, or by replacing both the spaces and operators with fully discrete counterparts \cite{Di-Pietro.Droniou.ea:20,Di-Pietro.Droniou:23*1}; see also the related works \cite{Codecasa.Specogna.ea:09,Codecasa.Specogna.ea:10,Pitassi:21,Beirao-da-Veiga.Lipnikov.ea:14,Bonelle.Ern:14,Bonelle.Di-Pietro.ea:15} dealing with the lowest-order case.
Both the virtual and fully discrete approaches lead to \emph{non-conforming} schemes, resulting in additional difficulties with respect to Finite Elements.

Poincaré inequalities on domains of general topology for the three-dimensional DDR complex of vector proxies of \cite{Di-Pietro.Droniou:23*1} have been recently proved in \cite{Di-Pietro.Hanot:24}.
Using operator-specific techniques, fully general Poincaré inequalities for the gradient and the divergence had already been obtained, respectively, in \cite[Theorem 3]{Di-Pietro.Droniou:23*1} and \cite{Di-Pietro.Droniou:21}.
A Poincaré inequality for the curl on contractible domains had been obtained in \cite[Theorem 20]{Di-Pietro.Droniou:21}.
We also cite here \cite{Di-Pietro:24} as a first example of Poincar\'e-type inequalities for two-dimensional de Rham complexes with enhanced regularity and, possibly, boundary conditions.

Very recently, Discrete de Rham (DDR) complexes of differential forms have been proposed in \cite{Bonaldi.Di-Pietro.ea:25}, generalizing both the constructions of \cite{Di-Pietro.Droniou:23*1} and of \cite{Beirao-da-Veiga.Brezzi.ea:18};
see also \cite{Droniou.Hanot.ea:24} for an extension to manifolds.
The purpose of this work is to establish discrete Poincar\'{e} inequalities for the complex studied in \cite[Section~3]{Bonaldi.Di-Pietro.ea:25}.
The proofs generalize and broaden the ideas of \cite{Di-Pietro.Hanot:24}.
The starting point is the Poincar\'e inequality for cochains stated in Lemma~\ref{lemma:topological.bound} below.
Given a polytopal mesh admitting a conforming simplicial submesh, this inequality is proved by constructing a mapping to lift polytopal cochains onto simplicial cochains, identifying the latter with conforming piecewise polynomial forms through the Whitney and de Rham maps in order to leverage continuous Poincar\'e inequalities, and concluding with local norms estimates.
The discrete Poincar\'e inequality for the arbitrary-order DDR complex is then proved in Theorem~\ref{thm:Poincare} using the Poincar\'e inequality for cochains to construct, given a discrete $k$-form, another discrete $k$-form with the same discrete exterior derivative and bounded by the latter.
Finally, in Corollary~\ref{cor:Poincaré} this result is bridged to a more standard form of the Poincar\'e inequality for discrete $k$-forms that lie in an orthogonal complement of the kernel of the exterior derivative.

The rest of the paper is organized as follows.
Section~\ref{sec:main.result} states the main results of the paper after briefly recalling the DDR construction.
The Poincar\'e inequality for cochains, stated and proved in Section~\ref{sec:poincare.cochains}, is then used in Section~\ref{sec:proof.poincare} to prove the main result.
The paper is completed by three appendices.
Appendix~\ref{app:polynomial.spaces} collects results on local polynomial spaces that generalize those proved in \cite{Di-Pietro.Droniou:20,Di-Pietro.Droniou:23*1} for polynomial functions.
Appendix~\ref{section:whitney} contains estimates for Whitney forms.
Finally, Appendix~\ref{section:algorithm} describes an algorithm for the computation of a basis for a subspace of simplicial cycles used in the proof of the Poincar\'e inequality for cochains.

%------------------------------------------------------------------------------%

\section{Main results}\label{sec:main.result}

\subsection{Setting}\label{sec:main.result:mesh}

\subsubsection{Polytopal mesh}

Let $k \in \{0, \dots, n\}$.
A \emph{(flat) $k$-cell} $f$ is a subset of a $k$-dimensional subspace of $\Real^n$ that is homeomorphic to a closed $k$-dimensional ball.
A \emph{polytopal mesh} $\Mh$ of a polytopal domain $\Omega \subset \Real^n$ is a finite collection of $k$-cells, for all $k\in\{0,\ldots,n\}$, such that the following properties hold:
\begin{enumerate}[label=(\roman*)]
\item The union of the $n$-cells cover $\Omega$;
\item For each $k\in\{1,\dots, n\}$, the boundary $\pf$ of each $k$-cell $f$ is a union of $(k-1)$-cells in $\Mh$;
\item The intersection of two distinct $k$-cells is either empty or is a union of lower-dimensional cells in $\Mh$;
\item\label{Mh:intersections} Given a $k$-cell $f$ and an $\ell$-cell $f'$ with $\ell \leq k$, such that $f \neq f'$ and $f\cap f' \neq \emptyset$, we have $f \cap f' \subset \partial f$.
\end{enumerate}

We denote by $\FM{k}(\Mh)$ the set of all $k$-cells of $\Mh$, and each of these $k$-cell is equipped by an orientation.
For any $k$-cell $f \in \FM{k}(\Mh)$, we define the \emph{polytopal submesh induced by $\Mh$ on $f$} to be $\Mh(f) \coloneqq \left\{ f' \in \Mh \st f' \subset f \right\}$.
It is easy to verify that $\Mh(f)$ is a polytopal mesh on $f$.
Moreover, $\Mh(\partial f)$ is also well-defined and is a polytopal mesh on $\partial f$.

A $k$-simplex $f$ is a $k$-cell that is obtained by applying an affine transformation (i.e, composition of a translation and a non-singular linear transformation) to the \emph{reference} $k$-simplex
\begin{equation*}
  \left\{ (x_0, \dots, x_k) \in \Real^{k+1} \st x_i \geq 0, \sum_{i=0}^k x_i = 1 \right\}.
\end{equation*}
We say that $\Mh$ is \emph{simplicial} if  it forms a simplicial mesh in the usual sense.
Given a polytopal mesh $\Mh$ of $\Omega$, we say that $\Ih$ is a \emph{simplicial submesh} of $\Mh$ if: (i) $\Ih$ is a simplicial mesh of $\Omega$; (ii) each $k$-cell of $\Mh$ is a union of $k$-simplices of $\Ih$.

The notion of mesh regularity is given by \cite[Definition~1.9]{Di-Pietro.Droniou:20}.
Namely, we say that a sequence of meshes $(\Mh)_h$ is regular if there exists $\rho \in (0,1)$, independent of $h$, 
and a sequence of simplicial submeshes $(\Ih)_h$ of $(\Mh)_h$, such that:
\begin{itemize}
\item For each $k \in \{1, \dots, \dtop\}$ and any $k$-cell $f$ of $\Ih$, it holds that $\rho h_f \leq r_f$, where $h_f$ denotes the diameter of $f$ and $r_f$ its inradius.
\item For any $k$-cell $f$ of $\Mh$, and all $k$-simplex $f'$ of $\Ih$ such that $f' \subset f$, 
  it holds that $\rho h_f \leq h_{f'}$.
\end{itemize}

To avoid the proliferation of constants, throughout the rest of the paper we write $a \lesssim b$ in place of $a \le C b$, with real number $C > 0$ independent of the mesh size and, for local inequalities on a $k$-cell $f$, also on $f$, but possibly depending on other parameters such as the domain, the mesh regularity parameter, the polynomial degree of the complex, etc.
We also write $a \simeq b$ for ``$a \lesssim b$ and $b \lesssim a$''.

The regularity requirements on the mesh imply that the number of simplices in a cell is bounded uniformly in $h$.
Specifically, for each $k \in \{1, \dots, \dtop\}$ and any $d$-cell $f$ of $\Mh$ with $d \in \{k, \dots, \dtop\}$, it holds that $\card(\FM{k}(\Ih(f)) \lesssim 1$.

\subsubsection{Discrete differential forms}

We give here a brief presentation, and refer to \cite[Section 2 and Appendix A]{Bonaldi.Di-Pietro.ea:25} for more details.
Let $\Mh$ be a polytopal mesh. 
Let $d \in \{0, \dots, \dtop\}$, and $f \in \FM{d}(\Mh)$ be a $d$-cell of $\Mh$.
In particular, $f$ is a (flat) $d$-dimensional manifold. 
For $k \in \{0, \dots, d\}$, we denote by $\Lambda^k(f)$ the set of differential $k$-forms on $f$. 
When a specific regularity is necessary, we prepend the appropriate space.
For instance, $C^0\Lambda^k(f)$ is the set of $k$-forms defined on $f$ and that are continuous. 

We recall that the inner product on the space of $k$-forms can be written
using the Gram determinant as:
\begin{equation}
  \langle \omega,\mu\rangle_f = \int_f \langle \omega_x, \mu_x\rangle \, \DIFF x
  = \int_f \det\big((\langle\omega_i(x),\mu_j(x)\rangle)_{i,j}\big) \, \DIFF x
  \qquad \forall \omega,\mu \in L^2\Lambda^k(f),
  \label{def:inner.product.gram}
\end{equation}
where $(\omega_j)_{j=1,\ldots,k}$ and $(\mu_j)_{j=1,\ldots,k}$ are 1-forms such that $\omega_x = \omega_1(x)\wedge\dots\wedge\omega_k(x)$
and $\mu_x = \mu_1(x)\wedge\dots\wedge \mu_k(x)$. The corresponding $L^2$-norm is
\[
\norm{f}{\omega} \coloneq \langle \omega, \omega\rangle_f^{\frac12}.
\]

We denote by $\star_f$ the Hodge star operator on a $d$-cell $f$ for $d\in\{0,\ldots,n\}$; it is an isomorphism $\Lambda^k(f)\to \Lambda^{d-k}(f)$.
For the sake of brevity, we shall simply denote it ``$\star$'' in the remaining of the paper, since its domain and codomain can be inferred from its argument. We note that, using this Hodge star operator, the inner product \eqref{def:inner.product.gram} can be written
\begin{equation}\label{def:inner.product.forms}
  \langle \omega, \mu \rangle_f \coloneq \int_f \omega \wedge \star \mu
  \qquad \forall \omega, \mu \in L^2 \Lambda^k(f).
\end{equation}

For each $f\in\Mh$ we fix a point ${x}_f$ such that $f$ is star-shaped with respect to a ball centered at $x_f$ and of radius $\gtrsim h_f$. 
We define the Koszul operator $\KOSZUL_f$ on $f$ to be the interior product against the identity vector field shifted at $x_f$.
Specifically, for $\omega \in \Lambda^k(f)$, $\KOSZUL_f\omega \in \Lambda^{k-1}(f)$ satisfies 
\[
(\KOSZUL_f\omega)_{{x}}({v}_1, \dots, {v}_{k-1}) 
= (\omega)_{{x}}({x}-{x}_f,{v}_1, \dots, {v}_{k-1}),
\]
for all ${x}\in f$, and all ${v}_1, \dots, {v}_{k-1}$ tangent vectors to $f$.
From hereon, and for the same reason that we dropped the index $f$ on the Hodge-star operator $\star$, we drop the index $f$ in the notation for the Koszul operator.

\subsection{Local spaces of forms with polynomial coefficients}

For any integer $r \geq 0$, we denote by $\PL{r}{k}(f)$ the space of $k$-forms on $f$ 
with polynomial coefficients of total degree less than or equal to $r$. 
By convention, we set $\PL{-1}{k}(f) \coloneq \left\{ 0 \right\}$.
For $k \geq 1$, the following decomposition of the polynomial space holds (see \cite[Eq.~(3.11)]{Arnold.Falk.ea:06}):
\[
\PL{r}{k}(f) = \DIFF \PL{r+1}{k-1}(f) \oplus \KOSZUL \PL{r-1}{k+1}(f).
\]
With this decomposition in mind, we define the space of \emph{trimmed} polynomials lowering by one the polynomial degree of the first component for $k \in \{1,\ldots,d\}$:
\begin{equation*}
  \begin{aligned}
    \PLtrim{r}{0}(f) &\coloneq \PL{r}{0}(f), \\
    \PLtrim{r}{k}(f) &\coloneq \DIFF \PL{r}{k-1}(f) \oplus \KOSZUL \PL{r-1}{k+1}(f) \quad\text{ if } k\in\{1,\ldots,d\}.
  \end{aligned}
\end{equation*}
Using the inner product \eqref{def:inner.product.gram}, we can define an associated orthogonal projector 
$\ltproj{k}{f} \st L^2\Lambda^k(f) \to \PLtrim{r}{k}(f)$ such that, for $\omega \in L^2\Lambda^k(f)$,
\[
\langle \ltproj{k}{f} \omega, \mu \rangle_f = \langle \omega, \mu \rangle_f, \quad \forall \mu \in \PLtrim{r}{k}(f).
\]

\subsection{The Discrete De Rham construction}\label{sec:ddr}

From this point on, we fix an integer $r \ge 0$ corresponding to the polynomial degree of the discrete de Rham complex.

\subsubsection{Spaces and interpolators}

For any $k \in \{0, \ldots, n\}$, the discrete counterpart $\uH{k}{h}$  of the space $H\Lambda^k(\Omega)$ is defined as
\[
\uH{k}{h} \coloneq \bigtimes_{d = k}^{n} \bigtimes_{f \in \FM{d}(\Mh)} \star^{-1}\PLtrim{r}{d-k}(f) .
\]
Notice that, compared with the original paper \cite{Bonaldi.Di-Pietro.ea:25}, we have taken $\star^{-1}\PLtrim{r}{d-k}(f)$ instead of $\PLtrim{r}{d-k}(f)$ as component space.
In the context considered here, based on a polytopal meshes of a domain of $\Real^n$, $\star$ is an isomorphism between spaces of polynomial forms and these choices equivalent up to an application of the Hodge star to each component.
In what follows, we will therefore directly invoke results from \cite{Bonaldi.Di-Pietro.ea:25} without specifying each time the adaptation (application of suitable Hodge star to the components) to the present setting.
When considering general Riemannian manifolds, however, $\star$ no longer maps polynomials into polynomials, and each choice leads to a different construction.
The one adopted here has the advantage to allow the generalization of the method to manifolds (as it is done in \cite{Droniou.Hanot.ea:24}), whereas the one of \cite{Bonaldi.Di-Pietro.ea:25} yields a more direct link with the DDR complex in vector calculus notation of \cite{Di-Pietro.Droniou:23*1}.

\subsubsection{Local construction}

For any dimension $d \in \{0, \dots, \dtop\}$, any face $f \in \FM{d}(\Mh)$, and any form degree $k \in \{0,\ldots,d\}$, we define the local interpolator $\uI{k}{f} : C^0\Lambda^k(f) \to \uH{k}{f}$ such that
\[
\uI{k}{f} \omega \coloneq \left(
\star^{-1}\ltproj{d'-k}{f'}\star \tr_{f'} \omega
\right)_{f' \in \FM{d'}(\Mh(f)), d' \in \{k,\ldots,d\}}
\qquad \forall \omega \in C^0\Lambda^k(f).
\]
We set, for the sake of brevity,
\begin{equation}\label{eq:oriented.product.pf}
  \langle \cdot,\cdot \rangle_{\pf} \coloneq \sum_{f'\in\FM{d-1}(\Mh(\partial f))}\orffp\langle \cdot,\cdot\rangle_{f'} ,
\end{equation}
where $\orffp\in\left\{ -1,+1 \right\}$ is the relative orientation of $f'$ with respect to $f$. 

Let a form degree $k \in \{0, \dots, \dtop\}$ and a vector of forms $\ul{\omega}_h \in \uH{k}{h}$ be fixed.
For all $d \ge k$ and all $f \in \FM{d}(\Mh)$, we define a local discrete exterior derivative $\DIFF^k_{r,f} \st \uH{k}{f} \rightarrow \star^{-1}\PL{r}{d-k-1}(f)$ and a local potential reconstruction $P^k_{r,f} \st \uH{k}{f} \rightarrow \star^{-1} \PL{r}{d-k}(f)$.
The definition is recursive on the cell dimension $d$:
\begin{itemize}
\item If $d = k$, we set
  \begin{equation} \label{eq:def.P.d=k}
    P^k_{r,f} \ul{\omega}_f \coloneq\, \omega_f \in \star^{-1}\PL{r}{0}(f).
  \end{equation}
\item For $d \in \{ k+1,\ldots,n\}$:
  \begin{enumerate}
  \item We first define the discrete exterior derivative such that
    \begin{multline} \label{eq:def.d}
      \langle \star \DIFF^k_{r,f} \ul{\omega}_f, \mu_f \rangle_f
      = (-1)^{k+1} \langle \star \omega_f , \DIFF \mu_f\rangle_f
      + \langle\star P^k_{r,\pf} \ul{\omega}_f , \tr_\pf \mu_f \rangle_\pf
      \\
      \forall \mu_f \in \PL{r}{d-k-1}(f),
    \end{multline}
    where $P^k_{r,\pf}\ul{\omega}_f$ is the piecewise polynomial form on $\pf$ obtained patching together the polynomials $(P^k_{r,f'}\ul{\omega}_{f'})_{f'\in\FM{d-1}(\Mh(\partial f))}$ (which are made available by the previous step).
  \item We then define the potential reconstruction that satisfies
    \begin{multline*} %% \label{eq:def.P}
      (-1)^{k+1} \langle \star P^k_{r,f} \ul{\omega}_f, \DIFF \mu_f + \nu_f \rangle_f
      \\
      = \langle \star \DIFF^k_{r,f} \ul{\omega}_f , \mu_f \rangle_f
      - \langle \star P^k_{r,\pf} \ul{\omega}_\pf , \tr_\pf \mu_f \rangle_\pf
      + (-1)^{k+1} \langle \star \omega_f , \nu_f \rangle_f
      \\
      \forall (\mu_f, \nu_f) \in \KOSZUL \PL{r}{d-k}(f) \times \KOSZUL \PL{r-1}{d-k+1}(f).
    \end{multline*}
  \end{enumerate}
\end{itemize}

\begin{remark}[Integral formulas]
  Expanding the definition \eqref{def:inner.product.forms} of the inner product on $f$ and each $f'\in\FM{d-1}(\Mh(f))$, recalling \eqref{eq:oriented.product.pf}, and using the relation $\sigma\wedge\star \xi=\star^{-1}\sigma\wedge\xi$ valid for any couple $(\sigma,\xi)$ of $\ell$-forms, we note that \eqref{eq:def.d} is equivalent to the following relation, which exposes a discrete Stokes formula:
  \begin{equation}\label{eq:def.d.int}
    \int_f \DIFF^k_{r,f} \ul{\omega}_f\wedge \mu_f
    = (-1)^{k+1} \int_f \omega_f \wedge \DIFF \mu_f
    + \sum_{f'\in\FM{d-1}(\Mh(f))}\orffp\int_{f'} P^k_{r,f'} \ul{\omega}_{f'} \wedge \tr_{f'} \mu_f.
  \end{equation}
\end{remark}

The global discrete spaces are connected by the discrete differential
$\ul{\DIFF}^k_{r,h}:\uH{k}{h}\to \uH{k+1}{h}$ obtained by projecting each local discrete exterior derivative on the corresponding component of $\uH{k+1}{h}$:
\begin{equation}\label{eq:def.ulDIFF}
  \ul{\DIFF}^k_{r,h} \ul{\omega}_h \coloneq \left(
  \star^{-1}\ltproj{d-k-1}{f}\star \DIFF^k_{r,f} \ul{\omega}_f
  \right)_{f \in \FM{d}(\Mh), d \in \{k+1, \dots, \dtop\}}.
\end{equation}

\subsubsection{Component norms}

For any form degree $k \in \{0, \dots, \dtop\}$ and any $\ul{\omega}_h \in \uH{k}{h}$, we define the following $L^2$-like norms recursively:
\begin{itemize}
\item For each $f \in \FM{k}(\Mh)$, we set
  \begin{equation}\label{eq:opn.f.k}
    \opn{f}{\ul{\omega}_f} \coloneq \norm{f}{\omega_f}.
  \end{equation}
\item For $d \in \{ k+1,\dots,\dtop\}$ and $f\in\FM{d}(\Mh)$,
  \begin{equation}\label{eq:opn.f}
    \opn{f}{\ul{\omega}_f} \coloneq \left(
    \norm{f}{\omega_f}^2 +
    h_f \sum_{f'\in\FM{d-1}(\Mh(f))} \opn{f'}{\omega_{f'}}^2
    \right)^{\frac12}.
  \end{equation}
\end{itemize}
The global norm is then defined as
\begin{equation} \label{eq:def.hnorm}
  \opn{h}{\ul{\omega}_h} \coloneq \left(
  \sum_{f\in\FM{\dtop}(\Mh)}\opn{f}{\ul{\omega}_h}^2
  \right)^{\frac12}.
\end{equation}

\begin{remark}[Equivalent component norms]\label{rem:component.norms:choice}
  A slightly different choice of component norms is made in \cite[Eq.~A.7]{Di-Pietro.Droniou.ea:24} (with $s=2$ in this reference).
  Accounting for the mesh regularity assumption (which gives, in particular, a uniform bound on the cardinality of $\FM{d-1}(\Mh(f))$ for each $f\in\FM{d}(\Mh)$), these choices are equivalent in terms of scaling with $h$, which entitles us to use the results of \cite[Appendix~A]{Di-Pietro.Droniou.ea:24}.
  A third equivalent choice of norm is obtained replacing \eqref{eq:opn.f} with
  \begin{equation}
    \opn{f}{\ul{\omega}_f} \coloneq \left(
    \sum_{i=k}^{\dtop} h_f^{\dtop - i} \sum_{f'\in\FM{i}(\Mh(f))} \norm{f'}{\omega_{f'}}^2
    \right)^{\frac12}.
    \label{eq:opn.f.expl}
  \end{equation}
\end{remark}

\begin{remark}[Properties of the potential reconstruction]
  Recalling \cite[Eq. (3.37)]{Bonaldi.Di-Pietro.ea:25}, we have that, for all integers $d \in \{0, \dots, \dtop\}$ and $k \in \{0,\ldots,d\}$, all $f \in \FM{d}(\Mh)$, and all $\ul{\omega}_f \in \uH{k}{f}$,
  \begin{equation}\label{eq:projP}
    \ltproj{d-k}{f} \star P^k_{r,f} \ul{\omega}_f = \star \omega_f .
  \end{equation}
  Moreover, following the same arguments as \cite[Eq.~(A.10)]{Di-Pietro.Droniou.ea:24}, we also get
  \begin{equation}\label{eq:Pb.m1}
    \norm{f}{\star P^k_{r,f} \ul{\omega}_f} \lesssim \opn{f}{\ul{\omega}_f}.
  \end{equation}
\end{remark}

\subsection{Poincar\'e inequalities for the Discrete de Rham complex}

The main result of this work is the following generic Poincar\'e inequality, which is valid for forms of any degree and domains of general topology.

\begin{theorem}[Poincaré inequality]\label{thm:Poincare}
  Let $k \in \{0,\ldots,\dtop-1\}$ and let $\ul\omega_h\in  \uH{k}{h}$.
  Then, there exists $\ul\tau_h \in \uH{k}{h}$ such that $\ul\DIFF^k_{r,h} \ul\tau_h = \ul\DIFF^k_{r,h} \ul\omega_h$ and
  \[
  \opn{h}{\ul\tau_h} \lesssim \opn{h}{\ul\DIFF^k_{r,h} \ul\omega_h}.
  \]
  Equivalently,
  \[
  \min_{\ul\phi_h\in \Ker \ul\DIFF^k_{r,h}} \opn{h}{\ul\omega_h + \ul\phi_h} \lesssim \opn{h}{\ul\DIFF^k_{r,h} \ul\omega_h}.
  \]
\end{theorem}

\begin{proof}
  See Section \ref{sec:proof.poincare}.
\end{proof}

This formulation of the Poincar\'e inequalities avoids any reference to an orthogonality condition.
In fact, it does not even rely on the existence of an associated inner product.
However, it is fully equivalent to the usual one, which is stated in the following corollary.

\begin{corollary}[Poincar\'e inequality on orthogonal complements]\label{cor:Poincaré}
  Let $\langle\cdot,\cdot\rangle_h$ be an inner product on $\uH{k}{h}$ such that its associated norm $\norm{h}{{\cdot}}$ satisfies $\norm{h}{\ul{\omega}_h} \simeq \opn{h}{\ul{\omega}_h}$ for all $\ul{\omega}_h \in \uH{k}{h}$, 
  and denote by $(\Ker \ul{\DIFF}_{r,h}^k)^\perp$ the orthogonal complement for this inner product. Then,
  Theorem \ref{thm:Poincare} is equivalent to the following statement:
  \begin{equation}
    \norm{h}{\ul{\omega}_h}\lesssim\norm{h}{\ul{\DIFF}_{r,h}^k\ul{\omega}_h}\qquad \forall \ul{\omega}_h\in(\Ker \ul{\DIFF}_{r,h}^k)^\perp.
    \label{eq:poincare.cor}
  \end{equation}
\end{corollary}

\begin{proof}
  First, we show that Theorem \ref{thm:Poincare} implies \eqref{eq:poincare.cor}.
  Let $\ul{\omega}_h\in(\Ker \ul{\DIFF}_{r,h}^k)^\perp$, and let $\ul{P}_h$ denote the orthogonal projector onto $(\Ker \ul{\DIFF}_{r,h}^k)^\perp$ for the scalar product $\langle\cdot,\cdot\rangle_h$.
  Theorem \ref{thm:Poincare} ensures the existence of $\ul{\phi}_h \in \Ker \ul{\DIFF}_{r,h}^k$ such that
  \begin{equation}\label{eq:orth.choice.phi}
    \opn{h}{\ul{\omega}_h+\ul{\phi}_h} \lesssim \opn{h}{\ul{\DIFF}_{r,h}^k\ul{\omega}_h}.
  \end{equation}
  By definition of $\ul{P}_h$, we have $\ul{P}_h \ul{\omega}_h = \ul{\omega}_h$ and $\ul{P}_h\ul{\phi}_h = \ul{0}$.
  Using the uniform equivalence of the two norms and the fact that $\norm{h}{{\ul{P}_h \cdot}} \leq \norm{h}{{\cdot}}$ since $\ul{P}_h$ is an orthogonal projector, we write
  \begin{equation*}
    \norm{h}{\ul{\omega}_h}
    = \norm{h}{\ul{P}_h(\ul{\omega}_h+\ul{\phi}_h)}
    \leq \norm{h}{\ul{\omega}_h+\ul{\phi}_h}
    \simeq \opn{h}{\ul{\omega}_h+\ul{\phi}_h}
    \overset{\eqref{eq:orth.choice.phi}}\lesssim \opn{h}{\ul{\DIFF}_{r,h}^k\ul{\omega}_h}
    \simeq \norm{h}{\ul{\DIFF}_{r,h}^k\ul{\omega}_h}. 
  \end{equation*}

  To prove the reverse implication, 
  let $\ul\omega_h\in \uH{k}{h}$ and let $\ul\psi_{h}$ be the orthogonal projection of $\ul\omega_h$ into $\Ker \ul{\DIFF}_{r,h}^k$.
  By construction, $\ul\omega_h - \ul\psi_h \in (\Ker \ul{\DIFF}_{r,h}^k)^\perp$ and $\ul\psi_h \in \Ker \ul\DIFF^k_{r,h}$.
  Moreover, it holds:
  \begin{multline*}
    \inf_{\ul\phi_h\in \Ker \ul\DIFF^k_{r,h}} \opn{h}{\ul\omega_h + \ul\phi_h} 
    \leq \opn{h}{\ul\omega_h-\ul\psi_h} \simeq \norm{h}{\ul\omega_h - \ul\psi_h}
    \\
    \overset{\eqref{eq:poincare.cor}}\lesssim \norm{h}{\ul{\DIFF}_{r,h}^k(\ul{\omega}_h-\ul\psi_h)}
    = \norm{h}{\ul{\DIFF}_{r,h}^k\ul{\omega}_h}
    \simeq \opn{h}{\ul\DIFF^k_{r,h} \ul\omega_h}. \qedhere
  \end{multline*}
\end{proof}

%------------------------------------------------------------------------------%

\section{A Poincar\'e inequality for cochains}\label{sec:poincare.cochains}

The proof of Theorem \ref{thm:Poincare} relies on a Poincar\'e inequality for cochains on $\Mh$. We first recall some concepts related to this notion, before stating and proving this inequality.

\subsection{Setting}

\subsubsection{Chain complex}

A \emph{$k$-chain} $w$ of $\Mh$ is a formal linear combination of $k$-cells
\begin{equation}
  \label{eq:chain}
  w = \sum_{f \in \FM{k}(\Mh)} w_f\, f,
\end{equation}
where the $w_f$ are real numbers.
The set of all $k$-chains, equipped with the natural addition and scalar multiplication, is a real vector space that we denote by $\chain{k}(\Mh)$.
For a $k$-chain $w \in \chain{k}(\Mh)$, its \emph{support} $\supp(w)$ is defined by
\begin{equation*}
  \supp(w) \coloneqq \left\{ f \in  \FM{k}(\Mh) \st w_f \neq 0 \right\}.
\end{equation*}
Observe that each $k$-cell $f \in \FM{k}(\Mh)$ is also a $k$-chain $f \in \chain{k}(\Mh)$, obtained by considering $w_f=1$ as the unique non-zero real number in the linear combination \eqref{eq:chain}.
Therefore, the set of all $k$-cells $\FM{k}(\Mh)$ forms a basis for $\chain{k}(\Mh)$, which we call the \emph{canonical basis} of $\chain{k}(\Mh)$.
We also consider the following two identifications: (i) we identify the $k$-cell $f$ taken with opposite orientation with the $k$-chain $-f \in \chain{k}(\Mh)$; (ii) we identify the (topological) boundary $\partial f$ with the \emph{boundary $(k-1)$-chain} $\bd{k} f \in \chain{k-1}(\Mh)$ defined by
\begin{equation}
  \bd{k} f \coloneqq \sum_{f' \in \FM{k-1}(\Mh(f))} \orffp\,f',
  \label{eq:def_bd}
\end{equation}
where $\orffp$ is equal to $+1$ if $f'$ has the orientation induced by that of $f$ and $-1$ otherwise.
The \emph{boundary operator} $\bd{k}: \chain{k}(\Mh) \to \chain{k-1}(\Mh)$ is defined by linearity on the space of $k$-chains by setting, for $w \in \chain{k}(\Mh)$ as in \eqref{eq:chain},
\[
\bd{k} w \coloneqq \sum_{f \in \FM{k}(\Mh)} w_f \, \bd{k} f.
\]
It can be checked that $\bd{k}\circ \bd{k+1}=0$ for $k \in \{1, \dots, n-1\}$.
As a consequence, the sequence
\[
\cdots \xrightarrow{\bd{k+1}} \chain{k}(\Mh) \xrightarrow{\bd{k}} \chain{k-1}(\Mh)
\xrightarrow{\bd{k-1}} \cdots,
\]
with $\chain{k}(\Mh) = \{0\}$ if $k<0$ or $k>n$, is a \emph{chain complex}.
Accordingly, we can also define $\bd{n+1} \coloneqq 0$ and $\bd{0} \coloneqq 0$ as the zero maps.

We define the subspace of \emph{$k$-cycles}
\[
\cycle{k}(\Mh) \coloneq \ker\bd{k} = \left\{
w \in \chain{k}(\Mh) \st \bd{k} w = 0
\right\} \subset \chain{k}(\Mh),
\]
and the space of \emph{$k$-boundaries}
\[
\boundary{k}(\Mh) \coloneq \Image\bd{k+1} = \left\{
\bd{k+1} w \st w \in \chain{k+1}(\Mh)
\right\} \subset \chain{k}(\Mh).
\]
Noticing that $\bd{k}\circ\bd{k+1}=0$ implies $\boundary{k}(\Mh)\subset \cycle{k}(\Mh)$, we then define the $k$-th \emph{homology space} as the quotient space
\[
H_k(\Mh) \coloneqq \cycle{k}(\Mh)/{\boundary{k}(\Mh)}.
\]

\subsubsection{Cochain complex}

A $k$-cochain $\lambda$ of $\Mh$ is a linear map $\lambda:\chain{k}(\Mh)\to \Real$.
Therefore, the set $\cochain{k}(\Mh)$ of all $k$-cochains is a vector space corresponding to the dual space of $\chain{k}(\Mh)$.
Given a $k$-chain $w \in \chain{k}(\Mh)$, we denote the value of $\lambda$ at $w$ by $\inner{\lambda}{w} \coloneqq \lambda(w)$.

From basic linear algebra, for each $k$-cell $f \in \FM{k}(\Mh)$, there exist a unique cochain $\hat f \in \cochain{k}(\Mh)$ such that
\[
\inner{\hat f}{f'}
= \delta_{\ffp}
\coloneq \begin{cases}
  1 & \text{ if $f=f'$,} \\
  0 & \text{ otherwise.}
\end{cases}
\]
The set $\big\{\hat f \in \cochain{k}(\Mh) \st f \in \FM{k}(\Mh) \big\}$ is the \emph{canonical basis} of $\cochain{k}(\Mh)$.
Accordingly, a $k$-cochain $\lambda \in \cochain{k}(\Mh)$ can be uniquely written as
\begin{equation}\label{eq:cochain}
  \lambda = \sum_{f \in \FM{k}(\Mh)} \lambda_f \,\hat f \text{ where } \lambda_f \coloneqq \inner{\lambda}{f}.
\end{equation}
We also define the inner product $\inner{\cdot}{\cdot}$ on $\FM{k}(\Mh)$ such that, for any $f, f' \in \FM{k}(\Mh)$, $\inner{f}{f'} \coloneq \inner{\hat f}{f'}=\delta_{\ffp}$, and we extend $\inner{\cdot}{\cdot}$ by linearity on $\chain{k}(\Mh)$.

On the space of $k$-cochains we define the \emph{coboundary operator} $\delta^k: \cochain{k}(\Mh) \to \cochain{k+1}(\Mh)$ as the \emph{adjoint of the boundary operator}, by requiring that, for every $k$-cochain $\lambda \in \cochain{k}(\Mh)$, the following identity holds:
\begin{equation}\label{eq:def.cobd}
  \inner{\cobd{k}\lambda}{w}
  = \inner{\lambda}{\bd{k+1}w}
  \qquad \forall w \in \chain{k+1}(\Mh).
\end{equation}
Using \eqref{eq:def.cobd} and $\bd{k}\circ \bd{k+1}=0$, it can be verified that $\cobd{k} \circ \cobd{k-1}=0$ for $k\in \{1,\ldots,n-1\}$.
%and that, for any $f \in \FM{k+1}(\Mh)$ and any $k$-cochain $\lambda$,
%\begin{equation}\label{eq:cobd.formula}
%\inner{\cobd{k}\lambda}{f} \overset{\eqref{eq:def.cobd}, \eqref{eq:def_bd}}= \sum_{f'\in \FM{k-1}(\Mh)} \orffp \inner{\lambda}{f'}.
%\end{equation}
As a consequence, the sequence
\begin{equation}
  \cdots \xrightarrow{\cobd{k-1}} \cochain{k}(\Mh) \xrightarrow{\cobd{k}} \cochain{k+1}(\Mh)
  \xrightarrow{\cobd{k+1}} \cdots,
  \label{eq:cochain.complex}
\end{equation}
with $\cochain{k}(\Mh) = \{0\}$ if $k<0$ or $k>n$, is a \emph{cochain complex}.
Accordingly, we also define $\cobd{n} \coloneqq 0$ as the zero map.
We define the space of \emph{$k$-coboundaries}
\[
\coboundary{k}(\Mh)
\coloneq \Image\cobd{k-1}
= \left\{
\cobd{k-1} \lambda \st \lambda \in \cochain{k-1}(\Mh)
\right\}\subset \cochain{k}(\Mh) .
\]

\subsection{Poincar\'e inequality for cochains}

The Poincar\'e inequality for cochains actually consists in finding, for each boundary cochain, a pre-image (by the coboundary operator) with coefficients in the canonical basis that can be estimated in terms of the coefficients of the coboundary cochain.
In what follows, $k$-chains $w \in \chain{k}(\Mh)$ and $k$-cochains $\lambda \in \cochain{k}(\Mh)$ will be identified whenever needed with the corresponding vectors of coefficients $(w_f)_{f \in \FM{k}(\Mh)}$ and $(\lambda_f)_{f \in \FM{k}(\Mh)}$ in their expansions \eqref{eq:chain} and \eqref{eq:cochain} on the canonical bases of $\chain{k}(\Mh)$ and $\cochain{k}(\Mh)$, respectively.

\begin{restatable}[Poincar\'e inequality for cochains]{lemma}{poincare}\label{lemma:topological.bound}
  Let $k \in \{0, \dots, \dtop-1\}$ and let a $(k+1)$-coboundary $\xi=(\xi_f)_{f \in \FM{k}(\Mh)} \in \coboundary{k+1}(\Mh)$ be given.
  Then, there exists a $k$-cochain $ \lambda=(\lambda_f)_{f \in \FM{k}(\Mh)} \in \cochain{k}(\Mh)$ such that $\delta^k\lambda = \xi$ and
  \begin{equation}\label{eq:topological.bound}
    \sum_{T\in\FM{\dtop}(\Mh)} h_T^{\dtop - 2k}
    \sum_{f\in\FM{k}(\Mh(T))} \lambda_{f}^2
    \lesssim \sum_{T\in\FM{\dtop}(\Mh)} h_T^{\dtop - 2k - 2}
    \sum_{f\in\FM{k+1}(\Mh(T))} \xi_{f}^2 .
  \end{equation}
\end{restatable}

\begin{proof}
  See Section \ref{sec:proof.poincare.cochains}.
\end{proof}

\begin{remark}[Interpretation of Lemma \ref{lemma:topological.bound} in the context of low-order compatible methods]
  Consider, for each $k\in \{0, \dots, n\}$, the canonical basis for $\cochain{k}(\Mh)$, and let $\mat D_k$ be the matrix associated with the linear map $\cobd{k}: \cochain{k}(\Mh) \to \cochain{k+1}(\Mh)$ for $k \in \{0, \dots, n-1\}$.
  By \eqref{eq:def.cobd}, $\mat D_k$ is precisely the \emph{incidence matrix between $(k+1)$-cells and $k$-cells}, which is routinely used in low-order compatible schemes like the Compatible Discrete Operator method~\cite{Bonelle.Ern:14,Bonelle.Di-Pietro.ea:15},
  the Discrete Geometric Approach~\cite{Codecasa.Specogna.ea:09, Codecasa.Specogna.ea:10, Pitassi:21}, or the Mimetic Finite Difference method~\cite{Beirao-da-Veiga.Lipnikov.ea:14}.
  In these works, the focus is on dimension $n=3$, and the coboundary operator $\cobd{k}$, for $k \in \{0, 1, 2\}$, acts as a topological counterpart of the continuous differential operators of the de Rham complex; in particular, $\cobd{0}$ acts as a gradient, $\cobd{1}$ as a curl, and $\cobd{2}$ as a divergence.
  For this reason, the following specific notation is also of common use: $\mat G \coloneqq \mat D_0$, $\mat C \coloneqq  \mat D_1$, and $\mat D \coloneqq  \mat D_2$.

  Applying the same reasoning of Corollary \ref{cor:Poincaré} to Lemma \ref{lemma:topological.bound}, we obtain the following mimetic version of the Poincar\'e inequality, which involves a suitable inner product and its associated norm:
  For all $\lambda = (\lambda_f)_{f \in \FM{k}(\Mh)} \in(\Ker \mat D^k)^\perp$, with orthogonal taken with respect to an appropriate mimetic product, it holds
  \begin{equation*}
    \opn{h}{\lambda}\lesssim\opn{h}{\mat D_k \lambda }.
  \end{equation*}
  This result generalizes the known mimetic Poincar\'e inequalities for the gradient and the curl (c.f.~\cite[Lemma 7.45 and Lemma 7.47]{Bonelle:14}) to domains with general topology in arbitrary dimension $n$.
\end{remark}

The proof of the Poincar\'e inequality on the cochains on $\Mh$ relies on a lifting of these cochains as cochains on a simplicial submesh, that we can then identify (through the Whitney and de Rham maps) to conforming piecewise polynomial forms.

\subsection{Mapping between polytopal and simplicial cochains}

Let $\Ih$ be a given simplicial submesh of $\Mh$.
In what follows, for clarity of notation, we use the symbols $f$ and $F$ to denote a generic $k$-cell and $k$-simplex of $\Mh$ and $\Ih$, respectively; the symbols $T$ and $S$ will also sometimes be used instead of $f$ and $F$, respectively, for the special case $k = n$.

The purpose of this section is to construct a graded map $(\op{k})_{k\in\{0,\ldots,n\}}$ that sends polytopal $k$-cochains onto simplicial $k$-cochains, and that is a \emph{cochain map} (i.e., it commutes with the coboundary operator).
We first introduce a \emph{selection operator} which, to each simplex, associates the polytopal element of smaller dimension that contains it;
then, we explain the principle of the construction (which requires a double induction);
finally, we establish the properties of this graded map: cochain property, and norm estimates.

\subsubsection{Selection map}

Each simplex in $\Ih$ is contained in one or more polytopal cells of $\Mh$; in the following lemma, we show that the lowest-dimensional of such cells is unique and that the simplex is actually contained in its relative interior.

\begin{lemma}[Lowest-dimensional containing polytopal cell]\label{lem:incidence}
  For each $F \in \FM{k}(\Ih)$, there exists a unique $d$-cell $f \in \FM{d}(\Mh)$ with $d \in \{k, \dots, n\}$ such that $F \subset f$ and $F \not \subset \partial f$.
\end{lemma}

\begin{proof}
  The set $\{\ell\in\{0,\ldots,n\}\,:\,\exists f\in\FM{\ell}(\Mh)\text{ s.t. }F\subset f\}$ is non-empty and contained in $\Natural$, so it has a minimum $d$. Let $f\in \FM{d}(\Mh)$ that contains $F$; we obviously have $d\ge k$ as a $d$-cell cannot contain a simplex of strictly lower dimension. If we had $F\subset \partial f$, since this boundary is the union of cells in $\FM{d-1}(\Mh)$, $F$ would be contained in some cell of $\FM{d-1}(\Mh)$, contradicting the definition of $d$. Hence, $F\subset f$ and $F\not\subset \partial f$.

We now prove the uniqueness of this $f$. Assume that we have a $d'$-cell $f' \neq f$ such that $F\subset f'$ and $F\not\subset \partial f'$. Without loss of generality we can assume that $d' \geq d$. Then, the intersection $f \cap f'$ is non-empty, since it contains $F$. However, by Point~\ref{Mh:intersections} in the definition of $\Mh$, this would imply $F \subset \partial f'$, contradicting the choice of $f'$.
\end{proof}

Thanks to Lemma \ref{lem:incidence}, we can define a function $g_k : \FM{k}(\Ih) \to \Mh$ which maps each $k$-simplex $F$ to the unique $d$-cell $f \coloneqq g_k(F) \in \Mh$ such that $F \subset f$ and $F \not \subset \partial f$.

\begin{proposition}[Partitions via lowest-dimensional containing cells]
  \label{prop:properties}
  For each $k \in \{0, \dots, n\}$ and $d \in \{k, \dots, n\}$, define the set of all $k$-simplices whose lowest-dimensional containing polytopal cell is of dimension $d$:
  \[
  \Pa{k}{d} \coloneqq \left\{
  F \in \FM{k}(\Ih) \st g_k(F) \in \FM{d}(\Mh)
  \right\}.
  \]
  Then, the following two properties hold:
  \begin{enumerate}[label=(\roman*)]
  \item The family $\{\Pa{k}{k}, \dots, \Pa{k}{n}\}$ forms a partition of $\FM{k}(\Ih)$, i.e., $\FM{k}(\Ih) = \bigsqcup_{d=k}^n \Pa{k}{d}$;

  \item If a $k$-simplex lies on the boundary of a $d$-cell, then the lowest-dimensional polytopal cell that contains this simplex is of dimension $<d$: for all $F\in\FM{k}(\Ih)$ such that $F \subset \partial f$ for some $f \in \FM{d}(\Mh)$, then $F \in \bigsqcup_{d'=k}^{d-1} \Pa{k}{d'}$.
  \end{enumerate}
\end{proposition}

  \begin{proof}
    (i) Is a direct consequence of the definition of $\Pa{k}{d}$ and of Lemma \ref{lem:incidence}.
    \smallskip\\
    (ii) This is a direct consequence of the reasoning in the proof of Lemma~\ref{lem:incidence}: if $F\subset \partial f$ for $f\in\FM{d}(\Mh)$ then $F$ is fully contained in some $f'\in\FM{d-1}(\Mh(f))$ and the minimal dimension of the cells that contains $F$ is thus $\le d-1$.
  \end{proof}

\subsubsection{Construction}

In what follows, we will consider that, for all $k \in \{0,\ldots,n\}$, every simplex in $F \in \Pa{k}{k}$ inherits the orientation of the unique $k$-cell $g_k(F) \in \FM{k}(\Mh)$ in which it is contained.

For each $k \in \{0, \dots, n\}$, we construct a map $\op{k} : \cochain{k}(\Mh) \to \cochain{k}(\Ih)$ which associates to each polytopal $k$-cochain $\lambda \in \cochain{k}(\Mh)$ a simplicial $k$-cochain $\op{k}(\lambda) \in \cochain{k}(\Ih)$.
Recalling the partition $\FM{k}(\Ih) = \bigsqcup_{d=k}^n \Pa{k}{d}$ and the fact that $\FM{k}(\Ih)$ is the canonical basis of $\chain{k}(\Ih)$, we define $\op{k}$ by specifying $\op{k}(\lambda) : \Pa{k}{d} \to \Real$ for all $d \in \{k, \dots, n\}$ and for a generic $\lambda$.

\paragraph{Principles driving the construction}

We have to define, for each $k\in\{0,\ldots,n\}$, each $\lambda\in\cochain{k}(\Mh)$, each $d\in\{k,\ldots,n\}$, and each $F\in \Pa{k}{d}$, the value of $\inner{ \op{k}(\lambda)}{F}$.
This has to be done in a way that ensures that $\op{k}$ commutes with the coboundary operator, so that commutation must be inherently embedded in the choice of this value.

The definition of $\inner{\op{k}(\lambda)}{F}$ is carried out through a double recursion.
The outer recursion proceeds on the cochain index $k$, decreasing from $k=n$ to $k=0$, while the inner recursion operates on the index $d$, increasing from $d=k$ to $d=n$.
For each pair of values $k$ and $d$, consider the pair of indices $(k,d)$, which serve to track the construction process of the map $I^k$ by identifying the specific set $P_k(d)$.
Table \ref{tab:recursion} illustrates the steps and how they are intertwined.

\begin{table}
  \centering
  \begin{tikzcd}[
      nodes in empty cells,
      column sep=small,
      row sep=1.3em,
      execute at end picture={
        \draw
        ([xshift=-2em, yshift=-1em]\tikzcdmatrixname-2-3.south) --
        ([xshift=-2em, yshift=0em]\tikzcdmatrixname-7-3.south);
        \draw
        ([xshift=2em,yshift=-1em]\tikzcdmatrixname-2-2.east) --
        ([xshift=0em,yshift=-1em,]\tikzcdmatrixname-2-6.east);
      }
    ]
    & & & \multi[cols, align=center]{r}{\text{Cochain index $k$ of $\lambda$}} & &\\[-1em]
    & & n & n-1 & n-2 & \cdots\\
    & 0 & \\
    \multi[rows, align=center]{d}{\rotatebox{90}{Dimension $d$ such that $F\in\Pa{k}{d}$}} & \vdots & \\
    & n-2 &  &  & P_{n-2}(n-2) \arrow{r}{} \arrow{d}{}& \cdots \\
    & n-1 &  & P_{n-1}(n-1) \arrow{r}{} \arrow{d}{}& P_{n-2}(n-1) \arrow{r}{} \arrow{d}{}& \cdots \\
    & n & P_n(n) \arrow{r}{} & P_{n-1}(n) \arrow{r}{} & P_{n-2}(n) \arrow{r}{}& \cdots
  \end{tikzcd}
  \\
  \caption{
    Illustration of the double recursion used to define the value $\inner{\op{k}(\lambda)}{F}$.
    The recursion proceeds first across the columns from left to right.
    Within each column, it advances row by row, starting from the diagonal element and moving down to the bottom.
    The arrows indicate the sets $P_k(d)$ that have already been used to define $\inner{\op{k}(\lambda)}{F}$ for every $F \in P_k(d)$.
    These sets are then used to define new values as the double recursion progresses.
    By the end of each column $k$, the mapping $\op{k}(\lambda) : \Pa{k}{d} \to \Real$ is fully defined.
  }
  \label{tab:recursion}
\end{table}

The diagonal steps $(k,d) = (k,k)$ in this table are easy. These steps correspond to the case where $F$ is a $k$-simplex contained in a $k$-cell $f\in\FM{k}(\Mh)$; since $\lambda\in\cochain{k}(\Mh)$, the real number $\inner{\lambda}{f}$ is fully defined and $\inner{ \op{k}(\lambda)}{F}$ can then be defined as a fraction (weighted by the relative volumes of $f$ and $F$) of $\inner{\lambda}{f}$.
This choice corresponds to \eqref{eq:top.dimension} (in the case $k=n$) and \eqref{eq:on.boundary} (in the case $k<n$) below.

Let us consider the first non-trivial step $(k,d) = (n-1,n)$ for $n > 1$, which is actually representative of all the other steps in the table. In this case, $F$ is an $(n-1)$-simplex contained in the interior of an $n$-cell $f$.
Let $z\in\cycle{n-1}(\Ih(f))$ be a $(n-1)$-cycle and notice that, since $f$ is homeomorphic to a closed $n$-dimensional ball, $z$ is a $(n-1)$-boundary, i.e., there exists $w \in \chain{n}(\Ih(f))$ such that $z = \bd{n} w$.
The commutation condition requires
\begin{equation}\label{eq:constraint}
  \inner{\op{n}(\cobd{n-1}\lambda)}{w}
  = \inner{\cobd{n-1}\op{n-1}(\lambda)}{w}
  \overset{\eqref{eq:def.cobd}}= \inner{\op{n-1}(\lambda)}{\bd{n} w}
  = \inner{\op{n-1}(\lambda)}{z}.
\end{equation}
In this expression, the left-hand side is fully known by the step $(k,d) = (n,n)$; moreover, expanding the right-hand side according to the decomposition $z = \sum_{F'\in\FM{n-1}(\Ih(f))} z_{F'}\,F'$ of $z$, the values $\inner{\op{n-1}(\lambda)}{F'}$ are also fully known by the step $(k,d)=(n-1,n-1)$ whenever $F'\in\supp(z)$ is contained in $\partial f$ (since, in this case, $F'\in\Pa{n-1}{n-1}$).

To exploit \eqref{eq:constraint} and define new values of $\op{n-1}(\lambda)$, we build upon the theory developed in \cite{Pitassi:22}. Specifically, consider cycles $z$ that belong to a complement $\cspace{n-1}(\Ih(f))$ of the space of cycles whose support is fully contained in $\partial f$; see~\eqref{eq:relative.decomposition} below.
Then, we select a particular $F\not\subset\partial f$ in the support of $z$, fix $\inner{\op{n-1}(\lambda)}{\hat{F}}=0$ for all $\hat{F}\in\supp(z)$ different from $F$, and determine the only remaining undefined value, $\inner{\op{n-1}(\lambda)}{F}$, so that the relation in \eqref{eq:constraint} is satisfied.

For the construction above to be valid, we must be cautious that, each time we select a $z$ to define some values $\inner{\op{n-1}(\lambda)}{F}$ (either $0$ or to ensure \eqref{eq:constraint} for $z$), we are not re-defining a value that is also fixed when selecting another $z'$.
This is achieved by first building (see Appendix \ref{section:algorithm}) a particular basis $(z_i)_i$ of the cycles $\cycle{n-1}(\Ih(f))$ that have some support in the interior of $f$, and such that each $z_i$ has some $(n-1)$-simplex $F_i$ in its support, that is interior to $f$ but not in the support of the other $z_j$ for $j\not= i$ (see \eqref{eq:prop.Fi} below). The process above is then applied to $z_i$ in order to fix $\inner{\op{n-1}(\lambda)}{F}=0$ for all $F\in\supp(z)$ interior to $f$ and different from $F_i$, and to set $\inner{\op{n-1}(\lambda)}{F_i}$ to ensure \eqref{eq:constraint} for $z=z_i$. Doing the same for another $z_j$ in the basis does not risk re-defining any value since, by choice of the basis, $F_i\not\in\supp(z_j)$ and, if $F\in\supp(z_i)\cap\supp(z_j)$, both the processes for $z_i$ and for $z_j$ set $\inner{\op{n-1}(\lambda)}{F}=0$.

The remaining work, carried out in Lemma \ref{lem:cochain.map} below, consists in checking that this local enforcement of the commutation condition \eqref{eq:constraint} actually leads to a global commutation property.

\paragraph{Detailed construction}

We now turn to the detailed construction of the map which is done, as indicated above, in a recursive way on the index $k \in \{n, \dots, 0\}$ (i.e., for $k$ ranging from $n$ down to $0$) and on the index $d \in \{k, \dots, n\}$ as described hereafter.
\begin{itemize}[leftmargin=1em]
\item \emph{Base case $k = n$.}
  For each $n$-simplex $F \in \Pa{n}{n}$, let $f = g_n(F) \in \FM{n}(\Mh)$ and set, for any $\lambda \in \cochain{k}(\Mh)$,
  \begin{equation}
    \inner{\op{k}(\lambda)}{F}
    \coloneq \frac{|F|}{|f|} \inner{\lambda}{f}.
    \label{eq:top.dimension}
  \end{equation}

\item \emph{Induction step $k \in \{n-1, \dots, 0\}$.}
  Assume that $\op{k+1}$ has been defined and let us define $\op{k}(\lambda)$ for a given $\lambda \in \cochain{k}(\Mh)$.

  \begin{itemize}[leftmargin=1em]
  \item \emph{Base case $d = k$.} For each $F \in \Pa{k}{k}$, let $f=g_k(F) \in \FM{k}(\Mh)$ and set
    \begin{equation}
      \inner{\op{k}(\lambda)}{F} \coloneq \frac{|F|}{|f|} \inner{\lambda}{f}.
      \label{eq:on.boundary}
    \end{equation}

  \item \emph{Induction step $d \in \{k+1, \dots, n\}$.}
    Assume that $\op{k}(\lambda) : \Pa{k}{d'} \to \Real$ has been defined for all $d' \in \{k, \dots, d-1\}$ and let us define $\op{k}(\lambda) : \Pa{k}{d} \to \Real$.
    Let $f \in \FM{d}(\Mh)$, and denote by $\Ih(f)$ the simplicial subcomplex induced by $\Ih$ on $f$.
    Let $\mc B_k = \{z_1, \dots, z_p \}$ with $p \in \mathbb N$ ($p=0$ means $\mc B_k = \emptyset$) be the family computed with Algorithm \ref{acyclic_algo} in Appendix \ref{section:algorithm}.
    By Lemma \ref{lem:basis}, $\cspace{k}(\Ih(f))=\mathrm{span}(\mc B_k)$ satisfies
    \begin{equation}
      \cycle{k}(\Ih(f)) = \cycle{k}(\Ih(\partial f)) \oplus \cspace{k}(\Ih(f)),
      \label{eq:relative.decomposition}
    \end{equation}
    and $\mc B_k$ is a basis of $\cspace{k}(\Ih(f))$ that has the following property (see Lemma \ref{lem:duality}): there exists a set of $p$ $k$-simplices $\mc F_k = \left\{ F_1, \dots, F_p \right\}$ of $\Ih(f)$ dual to $\mc B_k$ in the sense that
    \begin{equation}\label{eq:prop.Fi}
      \begin{aligned}
	F_i \not\subset \partial f&\quad \forall i \in \{1, \dots, p\},\\
	\inner{F_i}{z_j} = \delta_{ij}&\quad \forall i,j \in \{1, \dots, p\}.
      \end{aligned}
    \end{equation}
We also recall that, writing $z_i=\left(z_{i,F}\right)_{F \in \FM{k}(\Ih(f))}$, the following bound holds:
    \begin{equation}\label{eq:bound.zi}
    \sum_{F\in \FM{k}(\Ih(f))}z_{i,F}^2\lesssim 1.
    \end{equation}

    Let us now characterize $k$-cycles in terms of $k$-boundaries. Recalling that each $d$-cell $f$ is homeomorphic to a closed $d$-dimensional ball,\footnote{Letting $B^d$ be the $d$-ball, we have the following well-known characterization of its homology \cite{Allen:02}:
    \[
    H_k(B^d) \cong
    \begin{dcases}
      \mathbb Z & \text{if } k=0,\\[4pt]
      0 & \text{if } k \in \{1, \dots, d\}.
    \end{dcases}
    \]
    }
    we may write
    \begin{equation}
      \label{eq:homology.cases}
      \cycle{k}(\Ih(f))=
      \begin{dcases}
        \boundary{0}(\Ih(f)) \oplus \mathrm{span}\{ F^* \} & \text{if } k=0,\\[4pt]
        \boundary{k}(\Ih(f)) & \text{if } k \in \{1, \dots, d\},
      \end{dcases}
    \end{equation}
    where $F^* \in \FM{0}(\Ih(f))$ is an arbitrary $0$-cell on the boundary of $f$ (the existence of $F^*$ is guaranteed by the fact that $d>k=0$).
    The special case $k = 0$ can then be subsumed into the general setting in the spirit of \emph{reduced homology} \cite{Allen:02}.
    Specifically, forming the quotient
    \begin{equation}
      \widetilde Z_0(\Ih(f)) \coloneqq \cycle{0}(\Ih(f)) / \mathrm{span}\{ F^* \},
      \label{eq:substitution.a}
    \end{equation}
we obtain a natural isomorphism $\widetilde Z_0(\Ih(f)) \cong \boundary{0}(\Ih(f))$.
Substituting $\cycle{0}(\Ih(f)) \leftarrow \widetilde Z_0(\Ih(f))$, we may henceforth write $\cycle{k}(\Ih(f))=\boundary{k}(\Ih(f))$ for all $k\in\{0,\dots,d\}$, which is the intended effect of passing to reduced homology in degree zero.
Observing that the quotient $\widetilde Z_0(\Ih(f))$ in \eqref{eq:substitution.a} can be identified with $\mathrm{span}\{F-F^*\,:\,F\in\FM{0}(\Ih(f))\setminus\{F^*\}\}$, we also introduce the substitution
    \begin{equation}
      \mc B_0 \leftarrow \mc B_0 - F^* \coloneqq \{z_1 - F^*, \dots, z_p - F^*\}.
      \label{eq:substitution.b}
    \end{equation}
For each $z_i\in\mc B_k$, let $w_i$ be the $(k+1)$-chain given by Lemma \ref{lem:coeff.estimate}. We have
    \begin{equation}\label{eq:antiderivative.zi}
      \bd{k+1} w_i= z_i.
    \end{equation}
Furthermore, writing $w_i = \left(w_{i,F'}\right)_{F' \in \FM{k+1}(\Ih(f))}$, it follows from \eqref{eq:bound.zi} and \eqref{eq:bound.w.z} that
\begin{equation}\label{eq:bound.wi}
\sum_{F'\in\FM{k+1}(\Ih(f))} w_{i,F'}^2\lesssim 1.
\end{equation}

For each $z_i \in \mc B_k$, let $z_i|_{\partial f} \in \chain{k}(\Ih(\partial f))$ be the \emph{restriction of $z_i$ to $\partial f$}, that is, the $k$-chain such that $\supp(z_i|_{\partial f})=\supp(z_i) \cap \Ih(\partial f)$ and $\inner{F}{z_i|_{\partial f}} = \inner{F}{z_i}$ for each $k$-simplex $F \in \supp(z_i|_{\partial f})$.
    For each $F \in \Pa{k}{d}$ such that $f = g_k(F)$, we then let
    \begin{equation}
      \inner{\op{k}(\lambda)}{F} \coloneqq
      \begin{cases}
        \inner{\op{k+1} (\cobd{k} \lambda)}{w_i}
        - \inner{\op{k}(\lambda)}{z_i|_{\partial f}} &\text{if } F = F_i \in \{F_1, \dots, F_p\}, \\
        0 &\text{otherwise}.
      \end{cases}
      \label{eq:formula}
    \end{equation}
    Observe that, if $F \in \supp(z_i|_{\partial f})$, then $F \in \FM{k}(\Ih)$ and $F \subset \partial f$ since $z_i|_{\partial f} \in C_k(\Ih(\partial f))$.
    Applying Point (ii) in Proposition \ref{prop:properties}, we have $F \in \bigsqcup_{d'=k}^{d-1} \Pa{k}{d'}$, hence $\inner{\op{k}(\lambda)}{F}$ (and thus $\inner{\op{k}(\lambda)}{z_i|_{\partial f}}$) has already been defined by the induction hypothesis on $d$, while $\op{k+1}$ has been defined owing to the recurrence assumption.
    Hence, the recursive formula \eqref{eq:formula} is well-defined, i.e., its right-hand side is computable by induction for each $F \in \Pa{k}{d}$.
  \end{itemize}
\end{itemize}

\begin{example}[Definition of $\op{k}$ on a $3$-cell]
  We exemplify the definition of the map $\op{k}: \cochain{k}(\Mh) \to \cochain{k}(\Ih)$ for the specific case of the pyramidal $3$-cell $f$ shown in Figure~\ref{fig:pyramid}(a).

  Consider, for each $k \in \{3, \dots, 0\}$ and $d \in \{k, \dots, 3\}$, the pair of indices $(k,d)$, which keeps track of the construction process of the map $I^k$ by identifying the specific set $P_k(d)$.
  First, we notice that the cases for which $k=d$, i.e., $(k,d) \in \{(3,3), (2,2), (1,1), (0,0)\}$, are trivial, since formulas \eqref{eq:top.dimension}, \eqref{eq:on.boundary} are directly applied.
  Moreover, for this specific example, the cases $(k,d) \in \{(0,1), (0,3)\}$ are also trivial, since there are no $0$-cells of $\Ih(f)$ that are internal to any $1$-cell or $3$-cell of $\Ih(f)$.
  We thus focus on the remaining cases, i.e., $(k,d) \in \{(2,3), (1,2), (1,3), (0,2)\}$, for which, we show the construction of the sets $\mc F_k$ and $\mc B_k$ associated with the subspace $\cspace{k}(\Ih(f))$ of $\cycle{k}(\Ih(f))$ in Figure \ref{fig:pyramid}(b-d).
  \begin{figure}
    \begin{center}
      \includegraphics[scale=0.80]{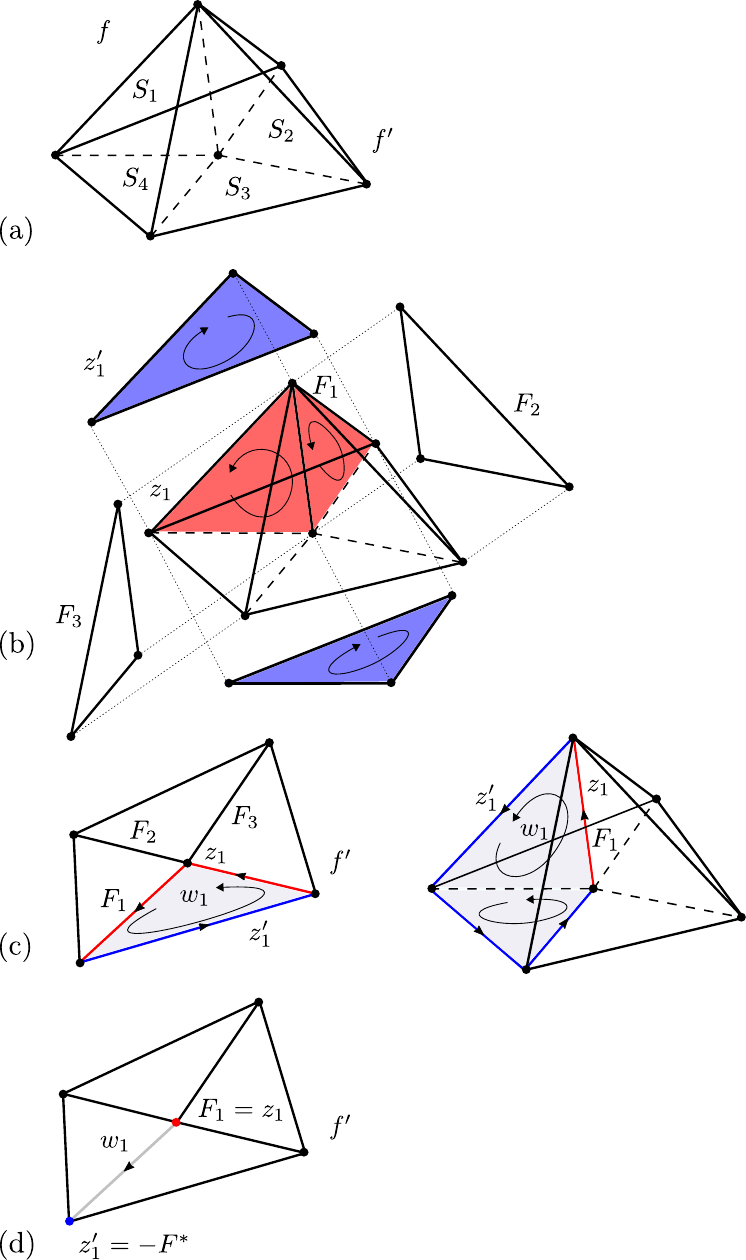}
    \end{center}
    \caption{(a) A simplicial subdivision $\Ih(f)$ of a pyramidal $3$-cell $f$ with a quadrilateral base $f'$; $\Ih(f)$ is made of four $3$-simplices so that $\FM{3}(\Ih(f)) = \{S_1,S_2, S_3, S_4\}$.
      (b) Case $(k,d)=(2,3)$ ($2$-simplices contained in the pyramid). A exploded view of the set $\mc F_2 = \{F_1, F_2, F_3\}$ of $2$-simplices associated with the basis $\mc B_2$ computed with Algorithm \ref{acyclic_algo} in Appendix \ref{section:algorithm}.
      In red and blue, the $2$-cycle $z_1 \in \mc B_2$ satisfying $\inner{F_i}{z_1} = \delta_{i1}$ for $i \in \{1,2,3\}$; the $3$-chain $w_1$ such that $\bd{3} w_1 = z_1$ is $w_1 = S_1$.
      In blue, the restriction $z_1|_{\partial f} \in \chain{2}(\Ih(f))$ of $z_1$ to $\partial f$. 
      (c) Case $(k,d) = (1,2)$ (left, corresponding to $1$-simplices contained in a polygonal $2$-cell) and $(k,d) = (1,3)$ (right, corresponding to $1$-simplices contained in the pyramid).
      On the left, $d=2$, and the set $\mc F_1 = \{F_1, F_2, F_3\}$ of 1-simplices associated with the basis $\mc B_1$ computed with Algorithm \ref{acyclic_algo}.
      In red and blue, the $1$-cycle $z_1 \in \mc B_1$ satisfying $\inner{F_i}{z_1} = \delta_{i1}$ for $i \in \{1,2,3\}$.
      In blue, the restriction $z_1|_{\partial f} \in \chain{1}(\Ih(f))$ of $z_1$ to $\partial f$, and in light grey the $2$-chain $w_1$ such that $\bd{2} w_1 = z_1$.
      On the right, $d=3$, and the set $\mc F_1 = \{F_1\}$.
      In red, blue, and light grey the corresponding $1$-chains $z_1$, $z_1|_{\partial f}$, and $2$-chain $w_1$ (supported by two $2$-simplices).
      (d) Case $(k,d) = (0,2)$.
      The set $\mc F_0 = \{F_1\}$ associated with the basis $\mc B_0$ computed with Algorithm \ref{acyclic_algo}, and the $0$-cell $F^* \in \FM{0}(\Ih(\partial f'))$.
      By the substitution in \eqref{eq:substitution.b}, the $0$-chain $z_1 \in \mc B_0$ is $z_1 \coloneqq F_1 - F^*$, and its restriction $z_1|_{\partial f'}$ to $\partial f'$ is the $0$-chain $z_1|_{\partial f'} = -F^*$. The $1$-chain $w_1$ such that $\bd{1}w_1 = z_1$ is then simply the $1$-simplex in $\Ih(f')$ connecting $F_1$ and $F^*$.
    }
    \label{fig:pyramid}
  \end{figure}

\end{example}

\subsubsection{Properties}

\begin{lemma}[Cochain map property for $(\op{k})_k$]
  \label{lem:cochain.map}
  The graded map $(\op{k})_k$, with $\op{k}: \cochain{k}(\Mh) \to \cochain{k}(\Ih)$ for $k \in \{0, \dots, n\}$, is a \emph{cochain map}, i.e., for all $k \in \{0, \dots, n\}$,
  \begin{equation}
    \cobd{k} \op{k} (\lambda) = \op{k+1} (\cobd{k} \lambda),  \qquad \forall \lambda \in \cochain{k}(\Mh),
    \label{eq:cochain.map}
  \end{equation}
  with the convention that $I^{n+1} \coloneqq 0$.
\end{lemma}

\begin{proof}
  We verify the formula \eqref{eq:cochain.map} by induction on $k\in \{n,n-1,\ldots,0\}$.\medskip\\
  \emph{1) Base case $k = n$.}
  Recalling the cochain complex \eqref{eq:cochain.complex}, we have $\cobd{n}=0$ and thus, by the convention on $\op{n+1}$, the relation \eqref{eq:cochain.map} holds.
  \medskip\\
  \emph{2) Induction step $k \in \{n-1, \dots, 0\}$.}
  Assume that \eqref{eq:cochain.map} holds for $k+1$ and let us prove it for $k$.
  To alleviate the notation, in the rest of the proof, for $\lambda\in \cochain{k}(\Mh)$, we set
  \begin{equation}\label{eq:shortcut.notation}
    \text{%
      $\tilde \lambda \coloneqq \op{k}(\lambda) \in \cochain{k}(\Ih)$ and $\tilde \xi \coloneqq \op{k+1} (\cobd{k} \lambda) \in \cochain{k+1}(\Ih)$.
    }
  \end{equation}
  We thus need to check that $\cobd{k} \tilde \lambda = \tilde \xi$, that is, recalling that $\FM{k+1}(\Ih)$ is the canonical basis of $\chain{k+1}(\Ih)$ and the relation \eqref{eq:def.cobd}, that
  \[
  \inner{\tilde \lambda}{\partial_{k+1} F}
  = \inner{\tilde \xi}{F}
  \qquad \forall F \in \FM{k+1}(\Ih).
  \]
  Since any $F \in \FM{k+1}(\Ih)$ is contained in $\Ih(f)$ for $f=g_k(F)\in\FM{d}(\Mh)$, it is sufficient to verify the following local version of the commutation property:
  For all $d \in \{k+1, \dots, n\}$ and all $f \in \FM{d}(\Mh)$,
  \begin{equation}\label{eq:check.condition.a}
    \inner{\tilde \lambda}{\partial_{k+1} w} = \inner{\tilde \xi}{w}
    \qquad \forall w \in \chain{k+1}(\Ih(f)).
  \end{equation}

  Let us now focus on the boundary map $\bd{k+1}$ appearing in \eqref{eq:check.condition.a}.
  By \eqref{eq:relative.decomposition}, we can find a basis $\{z_1, \dots, z_q\}$ of $\cycle{k}(\Ih(f))$ such that $\{z_1, \dots, z_p\}$ (with $p\le q$) is the basis $\mc B_k$ of $\cspace{k}(\Ih(f))$ obtained via Algorithm \ref{acyclic_algo} in Appendix \ref{section:algorithm} (and used in the definition of $\tilde \lambda$ in the recursive step), and $\{z_{p+1}, \dots, z_q\}$ is a basis of $\cycle{k}(\Ih(\partial f))$.
  Recall the substitution we performed in \eqref{eq:substitution.a} and \eqref{eq:substitution.b} for the sets $\cycle{0}(\Ih(f))$ and $\mc B_0$ via the $0$-cell $F^* \in \FM{0}(\Ih(\partial f))$ in \eqref{eq:homology.cases}, and notice also that the property \eqref{eq:prop.Fi} of the set $\mc F_0$ of $0$-simplices still holds for the new set $\mc B_0$ after the substitution \eqref{eq:substitution.b}, since $F^*\subset\partial f$.
  All of the above substitutions ensure that $\cycle{k}(\Ih(f)) = \boundary{k}(\Ih(f))$ whatever the value of $k$ (including $k = 0$), and we can find $(k+1)$-chains $\{w_1, \dots, w_q\}$ such that $z_i = \bd{k+1} w_{i}$ for each $i \in \{1, \dots, q\}$; we choose these chains to ensure that the first $p$ ones are those used in the definition of $I^k$ (see \eqref{eq:antiderivative.zi}).
  Accordingly, consider the direct decomposition
  \begin{equation}\label{eq:linear.map.decomposition}
    \chain{k+1}(\Ih(f))
    = \cycle{k+1}(\Ih(f)) \oplus \mathrm{span}\,\{w_1, \dots, w_q\},
  \end{equation}
  which follows by applying the First Isomorphism Theorem to the linear map $\bd{k+1}: \chain{k+1}(\Ih(f)) \to \boundary{k}(\Ih(f))$ (cf., e.g., \cite[Theorems~3.5 and~3.6]{Roman.Axler.ea:05}), after noticing that $\{z_1,\ldots,z_q\}$ is a basis of $\boundary{k}(\Ih(f))=\cycle{k}(\Ih(f))$, since $f$ is homeomorphic to a closed $d$-dimensional ball.

  Since $\tilde \xi \in \cochain{k+1}(\Ih)$, by induction hypothesis it holds
  \begin{equation}\label{eq:zero.cobd}
    \cobd{k+1} \tilde \xi
    \overset{\eqref{eq:shortcut.notation}}= \cobd{k+1} \op{k+1} (\cobd{k} \lambda)
    \overset{\eqref{eq:cochain.map}}= \op{k+2} \cobd{k+1}(\cobd{k} \lambda)
    = 0,
  \end{equation}
  where the last equality follows from the cochain complex property $\cobd{k+1}\circ \cobd{k} = 0$.
  Let now $w \in \cycle{k+1}(\Ih(f))$.
  Exploiting the fact that $\cycle{k+1}(\Ih(f)) = B_{k+1}(\Ih(f))$ since the $(k+1)$-th homology group is trivial (as $f$ is contractible), we can write $w = \bd{k+2} t$ with $t \in \chain{k+2}(\Ih(f))$.
  We then have
  \begin{gather}
    \inner{\tilde \lambda}{\bd{k+1}w}
    = \inner{\tilde\lambda}{0}=0, \label{eq:zero.a}
    \\
    \inner{\tilde\xi}{w}
    = \inner{\tilde\xi}{\bd{k+2}t}
    \overset{\eqref{eq:def.cobd}}= \inner{\cobd{k+1}\tilde\xi}{t}\overset{\eqref{eq:zero.cobd}}
    = 0,\label{eq:zero.b}
  \end{gather}
  meaning that \eqref{eq:check.condition.a} is trivially verified for the elements $w \in \cycle{k+1}(\Ih(f))$.
  Therefore, recalling the decomposition \eqref{eq:linear.map.decomposition} as well as the fact that $z_i=\bd{k+1}w_i$ for all $i\in\{1,\ldots,q\}$, the relation \eqref{eq:check.condition.a} follows if we show that
  \begin{equation}\label{eq:check.condition.b}
    \inner{\tilde \lambda}{z_{i}}
    =  \inner{\tilde \xi}{w_{i}}
    \qquad \forall i \in \{1, \dots, q\}.
  \end{equation}
  We check this condition by induction on the dimension $d \in \{ k+1,\ldots,\dtop\}$ of $f$.
  To this end, we will treat separately the cases $i \in \{1,\ldots,p\}$ (corresponding to $z_i \in \cspace{k}(\Ih(f))$) and $i \in \{p+1,\ldots,q\}$ (corresponding to $z_i \in \cycle{k}(\Ih(\partial f))$).
  \smallskip \\
  \emph{2a) Base case $d = k+1$.}
  Let $f \in \FM{k+1}(\Ih)$.
  In this case, $\cycle{k}(\Ih(\partial f))$ has dimension 1 and is spanned by the $k$-cycle
  \begin{align}\label{eq:zq}
    z_{p+1}
    &= \sum_{F' \in \FM{k}(\Ih(\partial f))} \orffp\, F'
    \\ \label{eq:zq'}
    &= \bd{k+1}\underbrace{\left(
      \sum_{F \in \FM{k+1}(\Ih(f))} F
      \right)}_{\eqcolon w_{p+1}},
  \end{align}
  where $f' \in \FM{k}(\Mh(\partial f))$ is such that $F' \subset f'$ and $\orffp$ denotes its orientation relative to $f$ (which is inherited by $F'$); see Figure~\ref{fig:scaling}.
  We can then write
  \begin{align}
    \begin{split}
      \inner{\tilde \lambda}{{z_{p+1}}}
      \overset{\eqref{eq:shortcut.notation},\,\eqref{eq:zq}}&=
      \sum_{F' \in \FM{k}(\Ih(\partial f))} \orffp \,  \inner{\op{k}(\lambda)}{F'}\\
      \overset{\eqref{eq:on.boundary}}&= \sum_{f' \in \FM{k}(\Mh(\partial f))} \orffp \left(
      \sum_{F' \in \FM{k}(\Ih(\partial f)), \, F' \subset f'}  \frac{|F'|}{|f'|} \inner{\lambda}{f'}
      \right)\\
      &= \sum_{f' \in \FM{k}(\Mh(\partial f))} \orffp \inner{\lambda}{f'}\\
      \overset{\eqref{eq:def_bd},\eqref{eq:def.cobd}}&= \inner{\cobd{k}\lambda}{f},
      \label{eq:first.term}
    \end{split}
  \end{align}
  where in the second line we have used the fact that, if $F' \in \FM{k}(\Ih(\partial f))$, then $F' \in \Pa{k}{k}$ by Point (ii) in Proposition \ref{prop:properties},
  while, to pass to the third line, we have used the fact that $\sum_{F' \in \FM{k}(\Ih(\partial f)), \, F' \subset f'} \frac{|F'|}{|f'|} =1$.

  \begin{figure}
    \begin{center}
      \includegraphics[scale=0.9]{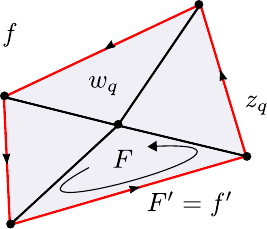}
    \end{center}
    \caption{
      Illustration of the $2$-chain $w_{p+1}$ (light grey) and $1$-chain $z_{p+1}$ (red) for a simplicial subdivision $\Ih(f)$ of a quadrilateral $2$-cell $f$.
      Note how the orientation of the $2$-simplex $F$ is the same as that of the $2$-cell $f$, as well as how $F$ and $f$ induce the same orientation on the unique $1$-cell $f' = F'$ (also, $1$-simplex in this case) which lies in their intersection $\partial F \cap \partial f = \{f'\}=\{F'\}$ so that $\epsilon_{FF'} = \orffp$.
    }
    \label{fig:scaling}
  \end{figure}

  Noticing that $F\in\FM{k+1}(\Ih(f))$ implies $g_{k+1}(F)=f$ and thus $F\in\Pa{k+1}{k+1}$, we can apply a similar reasoning as above and use the definition of $\op{k+1}(\cobd{k} \lambda)$ in \eqref{eq:on.boundary} but with $\cobd{k}\lambda$ in place of $\lambda$ to get
  \begin{align}
    \begin{split}
      \inner{\tilde \xi}{w_{p+1}}
      \overset{\eqref{eq:shortcut.notation},\,\eqref{eq:zq'}}&=
      \sum_{F \in \FM{k+1}(\Ih(f))}  \inner{\op{k+1}(\cobd{k}\lambda)}{F}\\
      &= \sum_{F \in \FM{k+1}(\Ih(f))}  \left(\frac{|F|}{|f|}\inner{\cobd{k} \lambda}{f} \right)\\
      &= \inner{\cobd{k} \lambda}{f}.
      \label{eq:second.term}
    \end{split}
  \end{align}
  Comparing \eqref{eq:first.term} and \eqref{eq:second.term}, we get \eqref{eq:check.condition.b} when $i = p + 1$.

  Take now $i \in \{ 1, \ldots, p \}$, i.e., $z_i$ is an element of the basis $\mc B_k$ of $\cspace{k}(\Ih(f))$.
  Recall that $z_i|_{\partial f} \in \chain{k}(\Ih(\partial f))$ is the restriction of $z_i$ to $\partial f$, so that ${\supp(z_i - z_i|_{\partial f})} \cap \FM{k}(\Ih(\partial f)) = \emptyset$.
  The properties \eqref{eq:prop.Fi} of the set of $k$-simplices $\mc F_k = \{F_1, \dots, F_p\}$ imply:
  (i) $F_i \not\subset \partial f$ and $F_i \in \supp(z_i)$ for any $i \in \{1, \dots, p\}$;
  (ii) $F_j \notin \supp(z_i)$ for any $j \neq i, j \in \{1, \dots, p\}$.
  By Property (i), $F_i \in \supp(z_i - z_i|_{\partial f})$, and so, in particular, the set $\supp(z_i - z_i|_{\partial f})$ is nonempty.
  Thus, if $F \in \supp(z_i - z_i|_{\partial f})$, then $F \subset f$ and $F \in \Pa{k}{k+1}$.
  By Property (ii), if $F \in \supp(z_i-z_i|_{\partial f})\setminus\{F_i\}$, we have $F\not\in\{F_1,\ldots,F_p\}$ and thus, applying the definition of $\tilde \lambda$ given by \eqref{eq:formula}, we infer $\inner{\tilde \lambda}{F} = 0$.
  Hence, $\inner{\tilde \lambda}{z_i-z_i|_{\partial f}-F_i} = 0$, so that $\inner{\tilde \lambda}{z_i|_{\partial f}}=\inner{\tilde \lambda}{z_i-F_i}$, and
  \begin{equation}\label{eq:on.internal}
    \inner{\tilde \lambda}{F_i}
    \overset{\eqref{eq:formula},\,\eqref{eq:shortcut.notation}}=\inner{\tilde \xi}{w_i}
    -\inner{\tilde \lambda}{z_i|_{\partial f}}
    =\inner{\tilde \xi}{w_i}
    -\inner{\tilde \lambda}{z_i - F_i},
  \end{equation}
  from which we infer
  \begin{align*}
    \inner{\tilde \lambda}{z_i}
    =\inner{\tilde \lambda}{z_i - F_i} + \inner{\tilde \lambda}{F_i}
    \overset{\eqref{eq:on.internal}}&= \inner{\tilde \xi}{w_i},
  \end{align*}
  and conclude the proof of \eqref{eq:check.condition.b} for $i \in \{1, \ldots, p\}$.
  \smallskip\\
  \emph{2b) Induction step $d \in \{k+2, \dots, n\}$.}
  Assume that \eqref{eq:check.condition.a} holds for all $d'$-cells with $d' \in \{k+1, \dots,  d-1\}$, and let $f \in \FM{d}(\Ih)$.

  Suppose first $z_i \in \cycle{k}(\Ih(\partial f))$.
  Since $f$ is homeomorphic to a closed $d$-dimensional ball, its boundary $\partial f$ is homeomorphic to a $(d-1)$-sphere\footnote{Letting $S^d$ be the $d$-sphere, we have the following well-known characterization of its homology \cite{Allen:02}:
  \[
  H_k(S^d) \cong
  \begin{dcases}
    \mathbb Z &\text{if } k \in \{0, d\}\\
    0 &\text{if } k \in \{1, \dots, d-1\}.
  \end{dcases}
  \]
  }.
  Since $k < d-1$, and recalling the substitution we performed in \eqref{eq:substitution.a} for the set $\cycle{0}(\Ih(f))$ using the $0$-cell $F^* \in \FM{0}(\Ih(\partial f))$ in \eqref{eq:homology.cases}, together with the fact that $\cycle{0}(\Ih(\partial f)) \subset \cycle{0}(\Ih(f))$, there exists a $(k+1)$-chain $w'_i \in \chain{k+1}(\Ih(\partial f))$ such that $\bd{k+1} w'_i = z_i$.
  
  Recalling that $z_i=\bd{k+1}w_i$ and that $f$ is contractible, we can write $w_i = w_i' + \bd{k+2} t$ with $t \in \chain{k+2}(\Ih(f))$.
  This observation implies that verifying \eqref{eq:check.condition.b} for $w'_i$ suffices, as \eqref{eq:check.condition.a} is trivially satisfied for $w=\bd{k+2}t \in \cycle{k+1}(\Ih(f))$, as shown by the relations \eqref{eq:zero.a} and \eqref{eq:zero.b}.
  
  Since every $F \in \supp(w'_i)$ is contained in $\Ih(f')$, where $f' = g_{k+1}(F) \in \Mh(\partial f)$ is a $d'$-cell with $d' \in \{k+1, \dots, d-1\}$, we can apply the induction hypothesis on $d$ to each $f'$.
  As a result, \eqref{eq:check.condition.a} is satisfied for each $w = F \in \supp(w'_i)$.
  Therefore, \eqref{eq:check.condition.b} is satisfied for $w=w_i'$, and thus for $w=w_i$, as required. 

  Now, suppose $z_i \in \cspace{k}(\Ih(f))$. In this case, to verify \eqref{eq:check.condition.b} is sufficient to repeat the same argument applied for the case $d = k+1$.
\end{proof}

\begin{lemma}[Estimate on $I^k$]
  Let $k \in \{0, \dots, n\}$ and $\lambda = (\lambda_f)_{f \in \FM{k}(\Mh)} \in \cochain{k}(\Mh)$, and set $\tilde \lambda = (\tilde \lambda_F)_{F \in \FM{k}(\Mh)} \coloneqq \op{k}(\lambda) \in \cochain{k}(\Ih)$.
  Then, for all $T\in \FM{n}(\Mh)$, the following uniform bound holds:
  \begin{equation}\label{eq:local.bound}
    \sum_{S \in \FM{n}(\Ih(T))} \sum_{F\in\FM{k}(\Ih(S))} \tilde \lambda_{F}^2
    \lesssim \sum_{f\in\FM{k}(\Mh(T))} \lambda_{f}^2 + \sum_{f \in \FM{k+1}(\Mh(T))} (\cobd{k} \lambda)_{f}^2 .
  \end{equation}
\end{lemma}

\begin{proof}
  Thanks to mesh regularity of $\Mh$, we have $\card\left(\{ S \in \FM{n}(\Ih(T)) \;:\; F \subset S\}\right)\le \card(\FM{n}(\Ih(T))) \lesssim 1$ for each $k \in \{0, \dots, n\}$, and it follows that
  \begin{equation}\label{eq:Ik:estimate:basic}
    \sum_{S \in \FM{n}(\Ih(T))} \sum_{F\in\FM{k}(\Ih(S))} \tilde \lambda_F^2
    \simeq \sum_{F \in \FM{k}(\Ih(T))} \tilde \lambda_F^2
    = \underset{\eqcolon A_1}{\underbrace{\sum_{F \in \Pa{k}{k},\, F \subset T}  \tilde \lambda_F^2 }}
    + \underset{\eqcolon A_2}{\underbrace{\sum_{d = k+1}^\dtop \sum_{F \in \Pa{k}{d}, \, F \subset T} \tilde \lambda_F^2},}
  \end{equation}
  where, to conclude, we have used the partition $\FM{k}(\Ih) = \Pa{k}{k} \sqcup \left( \bigsqcup_{d = k + 1}^n \Pa{k}{d} \right)$ of Proposition~\ref{prop:properties}.
  We now focus on bounding the terms $A_1$ and $A_2$ by the right-hand side of \eqref{eq:local.bound}. This is done by induction on $k \in \{n,\ldots,0\}$.
  \medskip\\
  \emph{Base case $k = n$.}
  In this case, $A_2=0$ and we just have to bound $A_1$.
  Applying the definition of $\tilde \lambda_S$ given by \eqref{eq:top.dimension}, and observing that $\frac{|F|}{|f|} \le 1$ for each $F\in\FM{n}(\Ih(f))$, we have
  $$
  \sum_{F \in \FM{n}(\Ih(f))} \left(\frac{|F|}{|f|} \right)^2\le
  \sum_{F \in \FM{n}(\Ih(f))} \frac{|F|}{|f|} = 1,
  $$
  from which we readily infer, using $\card(\FM{n}(\Ih(f))) \lesssim 1$ (consequence of mesh regularity),
  \begin{equation*}
    %\label{eq:conclusion.a1.a}
    A_1 = \sum_{F \in \FM{n}(\Ih(f))} \left(\frac{|F|}{|f|} \right)^2 \lambda_{{f}}^2 \lesssim \lambda_{f}^2.
  \end{equation*}
  This yields \eqref{eq:local.bound} since $\FM{n}(\Mh(f))=\{f\}$.
  \medskip\\
  \emph{Induction step $k \in \{n-1, \dots, 0\}$.}
  Assume that the inequality \eqref{eq:local.bound} holds for $k+1$ and let us prove it for $k$.
  \smallskip\\
  \emph{Term $A_1$.} Using the definition of $\tilde \lambda_F$ given by formula \eqref{eq:on.boundary} and reasoning as above, we have
  \begin{equation}
    \label{eq:conclusion.a1.b}
    A_1 = \sum_{f \in \FM{k}(\Mh(T))} \sum_{F \in \FM{k}(\Ih(f))} \left(\frac{|F|}{|f|} \right)^2 \lambda_{f}^2 \lesssim \sum_{f\in\FM{k}(\Mh(T))} \lambda_{f}^2.
  \end{equation}
  \smallskip\\
  \emph{Term $A_2$.}
  Let $f \in \FM{d}(\Mh(T))$ for $d \in \{k+1, \dots, n\}$.
  Let $\tilde \zeta = (\tilde \zeta_F)_{F \in \FM{k}(\Ih(f))} \in \cochain{k}(\Ih(f))$ and $t= \left(t_F\right)_{F \in \FM{k}(\Ih(f))} \in \chain{k}(\Ih(f))$.
Using a Cauchy--Schwarz inequality, we see that
  \begin{equation}
    \label{eq:cochain.bound}
    |\inner{\tilde \zeta}{t}|^2
    =\left| \sum_{F \in \FM{k}(\Ih(f))} t_F \, \tilde \zeta_F \right|^2
\lesssim 
     \sum_{F\in\FM{k}(\Ih(f))}t_F^2 \times \sum_{F \in \FM{k}(\Ih(f))} \tilde \zeta_F^2.
  \end{equation}

  Let $F \in \FM{k}(\Ih(f)) \cap \Pa{k}{d}$ and recall the definition \eqref{eq:formula} of $\tilde \lambda_F$.
  Set $\tilde \xi\coloneq \op{k+1}(\cobd{k}\lambda)$, 
  apply the estimate \eqref{eq:cochain.bound} to both terms $\inner{\tilde \lambda}{z_i|_{\partial f}}$ and $\inner{\tilde \xi}{w_i}$, and recall the bounds \eqref{eq:bound.zi} and \eqref{eq:bound.wi} on $z_i$ and $w_i$.
  This leads to
  \begin{equation}
    \label{eq:recursive.a}
    \sum_{F \in \Pa{k}{d}, \, F \subset f} \tilde \lambda_{F}^2
    \lesssim \sum_{F \in \FM{k}(\Ih(\partial f))} \tilde \lambda_{F}^2
    +\sum_{F\in\FM{k+1}(\Ih(f))} \tilde \xi_{F}^2 \qquad \forall f \in \FM{d}(\Mh(T)).
  \end{equation}
  Observe that, by Point (ii) in Proposition \ref{prop:properties}, the first term in the right-hand side of \eqref{eq:recursive.a} can be equivalently written as
  \[
  \sum_{d' = k}^{d-1}
  \sum_{f' \in \FM{d'}(\Mh(f))} \left( \sum_{F \in \Pa{k}{d'}, \,F \subset f'} \tilde\lambda_{F}^2 \right),
  \]
  and, in the case $d=k+1$, it holds for this term
  \[
  \sum_{f' \in \FM{k}(\Mh(f))} \left( \sum_{F \in \Pa{k}{k}, \,F \subset f'} \tilde\lambda_{F}^2 \right)
  \overset{\eqref{eq:on.boundary}}\lesssim
  \sum_{f'\in\FM{k}(\Mh(f))} \lambda_{f'}^2.
  \]
  Hence, via a simple induction argument on the dimension $d$ involving the bound \eqref{eq:recursive.a} and the mesh regularity assumption (to uniformly bound the cardinalities of $\FM{k}(\Mh(f))$ and $\FM{k}(\Ih(f))$), we conclude
  \begin{equation}
    \label{eq:conclusion.a2.a}
    A_2
    = \sum_{d = k+1}^\dtop  \sum_{f \in \FM{d}(\Mh(T))} \left( \sum_{F \in \Pa{k}{d},\, F \subset f} \tilde \lambda_F^2 \right)
    \lesssim \sum_{f \in \FM{k}(\Mh(T))} \lambda_{f}^2
    +\sum_{F\in\FM{k+1}(\Ih(T))} \tilde \xi_{F}^2.
  \end{equation}

  Recalling that $\tilde{\xi}=\op{k+1}(\cobd{k}\lambda)\in\cochain{k+1}(\Mh)$ we can apply the induction hypothesis for $k+1$ with $\lambda$ replaced by $\cobd{k}\lambda$, which states that this cochain satisfies \eqref{eq:local.bound}.
  Since $\cobd{k+1}(\cobd{k}\lambda)=0$ by complex property, this gives
  \begin{equation}
    \label{eq:conclusion.a2.b}
    \sum_{F \in \FM{k+1}(\Ih(T))} \tilde \xi_{F}^2
    \lesssim \sum_{f\in\FM{k+1}(\Mh(T))} (\cobd{k}\lambda)_f^2.
  \end{equation}
  Using \eqref{eq:conclusion.a1.b} to estimate $A_1$ and plugging \eqref{eq:conclusion.a2.b} into \eqref{eq:conclusion.a2.a} to estimate $A_2$ in \eqref{eq:Ik:estimate:basic} concludes the proof of the induction step.
\end{proof}

\subsubsection{Left inverse}

For $k \in \{0, \dots, n\}$, let $(\iop{k})_k$ be the graded map with $\iop{k} : \cochain{k}(\Ih) \to \cochain{k}(\Mh)$ such that, for any $\tilde \lambda \in \cochain{k}(\Ih)$, it holds
\begin{equation}
  \label{eq:reconstruction}
  \inner{\iop{k}(\tilde \lambda)}{f} \coloneqq \sum_{F \in \FM{k}(\Ih(f))} \inner{\tilde \lambda}{F} \qquad \forall f \in \FM{k}(\Mh).
\end{equation}

\begin{lemma}[Left inverse of $\op{k}$]
  \label{lem:reconstruction}
  The graded map $(\iop{k})_k$ satisfies the following properties:
  \begin{enumerate}[label=(\roman*)]
  \item $\iop{k}$ is a left inverse of $\op{k}$;
  \item $(\iop{k})_k$ is a cochain map;
  \item For all $T\in \FM{n}(\Mh)$ and $\tilde\lambda = (\tilde \lambda_F)_{F \in \FM{k}(\Ih)}\in \cochain{k}(\Ih)$, setting $\lambda = (\lambda_f)_{f \in \FM{k}(\Mh)}\coloneqq \iop{k}(\tilde\lambda) \in \cochain{k}(\Mh)$, the following bound holds:
    \begin{equation}
      \label{eq:local.bound.b}
      \sum_{f\in\FM{k}(\Mh(T))} \lambda_{f}^2
      \lesssim  \sum_{S \in \FM{n}(\Ih(T))} \sum_{F\in\FM{k}(\Ih(S))} \tilde \lambda_{F}^2.
    \end{equation}
  \end{enumerate}
\end{lemma}

\begin{proof}
  (i) It suffices to write, for all $f\in\FM{k}(\Mh)$,
  \begin{equation*}
    \inner{\iop{k}(\op{k}(\lambda))}{f}
    \overset{\eqref{eq:reconstruction}}=\sum_{F \in \FM{k}(\Ih(f))} \inner{\op{k}(\lambda)}{F}
    \overset{\eqref{eq:top.dimension}, \eqref{eq:on.boundary}}=\inner{\lambda}{f},
  \end{equation*}
  where we have additionally noticed that $g_k(F) = f$ for all $F \in \FM{k}(\Ih(f))$ in the last step.
  \medskip\\
  (ii) Let $\tilde\lambda$ and $\lambda$ as in Point (iii) and set
  $\tilde \xi \coloneqq \cobd{k} \tilde \lambda$ and $\xi \coloneqq \iop{k+1} (\tilde \xi)$.
  We have to show that $\cobd{k} \lambda=\xi$.
  This is done by writing, for any $f \in \FM{k+1}(\Mh)$,
  \begin{equation*}
    \begin{aligned}
      \inner{\cobd{k} \lambda}{f} \overset{\eqref{eq:def.cobd},\eqref{eq:def_bd},\eqref{eq:reconstruction}}&= \sum_{f'\in \FM{k}(\Mh(\partial f))} \orffp \sum_{F'\in\FM{k}(\Ih(f'))} \inner{\tilde \lambda}{F'}
      \\
      &= \sum_{F \in \FM{k+1}(\Ih(f))} \sum_{F' \in \FM{k}(\Ih(\partial F))} \epsilon_{FF'} \inner{\tilde \lambda}{F'}
      \\
      \overset{\eqref{eq:def_bd}, \eqref{eq:def.cobd}}&= \sum_{F \in \FM{k+1}(\Ih(f))} \inner{\tilde \xi}{F}
      \\
      \overset{ \eqref{eq:reconstruction}}&= \inner{\xi}{f},
    \end{aligned}
  \end{equation*}
  where, in the second equality, we have used the assumption that the orientation on $\Ih$ is the one induced by the maps $g_\bullet:\Ih\to\Mh$ (which implies $\epsilon_{FF'}=\orffp$ if $f=g_{k+1}(F)$ and $f'=g_k(F')$ in these sums), together with the relation
  \[
  \sum_{F \in \FM{k+1}(\Ih(f))} \sum_{F' \in \FM{k}(\Ih(\partial F)),\, F' \not\subset \partial f} \epsilon_{FF'} \tilde \lambda_{F'}=0,
  \]
  which translates the fact that contributions on $k$-simplices that do not lie on the boundary of $f$ cancel out since $\epsilon_{F_1F'}+\epsilon_{F_2F'}=0$ whenever $F_1,F_2$ are the $(k+1)$-simplices in $\Ih(f)$ on each side of $F'$.
  \medskip\\
  (iii) Is a direct consequence of the definition \eqref{eq:reconstruction} and of the regularity assumption on $\Mh$.
\end{proof}

\subsection{Proof of the Poincar\'e inequality for cochains}\label{sec:proof.poincare.cochains}

We can now prove Lemma \ref{lemma:topological.bound}, which we recall here for the sake of legibility.

\poincare*

\begin{proof}
  Let $\theta \in \cochain{k}(\Mh)$ be such that $\cobd{k} \theta = \xi$, and consider the following extensions to simplicial cochains:
  \begin{equation}\label{eq:topological.bound:shortcuts}
    \text{%
      $\tilde \xi \coloneq \op{k+1}(\xi)$ and $\tilde \theta \coloneqq \op{k}(\theta)$.
    }
  \end{equation}
  By Lemma \ref{lem:cochain.map}, $(I^k)_k$ is a cochain map so we have
  \begin{equation}
    \label{eq:passage}
    \cobd{k} \tilde \theta = \cobd{k}\op{k}(\theta) = \op{k+1} (\cobd{k} \theta) = \tilde \xi.
  \end{equation}
  For $k \in \{0, \dots n\}$, denote by $R_h^k : \mathfrak W_k(\Ih) \to \cochain{k}(\Ih)$ the \emph{De Rham map} restricted to the space $\mathfrak W^k(\Ih)$ of Whitney $k$-forms (see \cite{Arnold:18} and Appendix \ref{section:whitney}), and by $W_h^k : \cochain{k}(\Ih) \to \mathfrak W_k(\Ih)$ the \emph{Whitney map} on the simplicial complex $\Ih$ \cite{Dodziuk:76}.
  For each $\tilde \zeta = (\tilde \zeta_F)_{F \in \FM{k}(\Ih)} \in \cochain{k}(\Ih)$, $W_h^k(\tilde \zeta)$ is explicitly written as
  \begin{equation}\label{eq:def.psi_h}
    W_h^k(\tilde \zeta) \coloneqq \sum_{S\in\FM{\dtop}(\Ih)} \sum_{F\in\FM{k}(\Ih(S))}  \tilde \zeta_F \phi^k_{S,F}.
  \end{equation}
  Crucially, the graded map $(W_h^k)_k$ is a cochain isomorphism between the simplicial cochain complex and the Whitney form complex supported on $\Ih$, and $R_h^k$ is the inverse of $W_h^k$.

  Let
  \begin{equation}\label{eq:topological.bound:.txi.tth}
    \text{%
      $\tilde\xi_h \coloneqq W_h^{k+1}(\tilde \xi)$\quad and\quad
      $\tilde\theta_h \coloneqq W_h^k(\tilde \theta)$.
    }
  \end{equation}
  Thanks to the cochain map property of $W_h^k$, we have $\DIFF^k \tilde\theta_h = \DIFF^k W_h^k(\tilde \theta) = W_h^{k+1}(\cobd{k} \tilde \theta) \overset{\eqref{eq:passage}}= \tilde\xi_h$.
  Therefore, $\tilde\xi_h \in \Image \DIFF^k$. Invoking the continuous Poincaré inequality \cite[Theorem~5.11]{Arnold.Falk.ea:06}, then the proof of Corollary \ref{cor:Poincaré} to retrieve a formulation of this inequality in the spirit of Theorem \ref{thm:Poincare}, we infer the existence of $\tilde\lambda_h \in \mathfrak W_k(\Ih)$ such that
  \begin{equation} \label{eq:prop.psi.phi}
    \DIFF^k \tilde\lambda_h = \tilde\xi_h \quad\text{ and } \quad\norm{L^2 \Lambda^k(\Omega)}{\tilde\lambda_h} \lesssim \norm{L^2 \Lambda^{k+1}(\Omega)}{\tilde\xi_h}.
  \end{equation}
  Letting $\tilde \lambda \coloneqq R_h^k(\tilde\lambda_h)$ and using the cochain map property of $R_h^k$, we find
  \begin{equation}
    \label{eq:def.lambda}
    \cobd{k} \tilde \lambda
    = \cobd{k} R_h^k(\tilde\lambda_h) = R_h^{k+1} ( \DIFF^k  \tilde\lambda_h)
    \overset{\eqref{eq:prop.psi.phi}}= R_h^{k+1} \tilde \xi_h
    \overset{\eqref{eq:topological.bound:.txi.tth}}= R_h^{k+1} W_h^{k+1}(\tilde \xi) = \tilde \xi,
  \end{equation}
  where the conclusion follows since $R_h^{k+1}$ is the inverse of $W_h^{k+1}$.
  Letting $\lambda \coloneqq \iop{k}(\tilde \lambda)$, we then have
  \begin{equation*}
    \cobd{k} \lambda
    = \cobd{k}\iop{k}(\tilde \lambda)
    = \iop{k+1} ( \cobd{k} \tilde \lambda)
    \overset{\eqref{eq:def.lambda}}= \iop{k+1} \tilde \xi
    \overset{\eqref{eq:topological.bound:shortcuts}}= \iop{k+1} I^{k+1}(\xi)
    = \xi,
  \end{equation*}
  where we have used, respectively, Points (ii) and (i) in Lemma \ref{lem:reconstruction} to infer the second and last equalities.

  It remains to prove that $\lambda = (\lambda_f)_{f \in \FM{k}(\Mh)}$ satisfies \eqref{eq:topological.bound}. Decomposing $\tilde \lambda_h=W_h^k(\tilde\lambda)$ according to \eqref{eq:def.psi_h}, noticing that $\sum_{S \in \FM{n}(\Ih)} \bullet = \sum_{T \in \FM{n}(\Mh)} \sum_{S \in \FM{n}(\Ih(T))} \bullet$, and observing that each $\phi_{S,F}$ is only supported in $S$, we write (equations \eqref{eq:Whitney.topo} and \eqref{eq:Whitney.norm} are stated and proved in Appendix~\ref{section:whitney})
  \begin{equation}\label{eq:norm.phi_h}
    \begin{aligned}
      \norm{L^2 \Lambda^k(\Omega)}{\tilde\lambda_h}^2
      &= \sum_{T\in\FM{\dtop}(\Mh)} \sum_{S\in\FM{\dtop}(\Ih(T))} \Norm{S}{\sum_{F\in\FM{k}(\Ih(S))}\tilde\lambda_F \phi^k_{S,F}}^2 \\
      \overset{\text{\eqref{eq:Whitney.topo}}}&\simeq \sum_{T\in\FM{\dtop}(\Mh)} \sum_{S\in\FM{\dtop}(\Ih(T))} \sum_{F\in\FM{k}(\Ih(S))}\tilde\lambda_F^2 \norm{S}{\phi^k_{S,F}}^2 \\
      \overset{\text{\eqref{eq:Whitney.norm}}}&\simeq \sum_{T\in\FM{\dtop}(\Mh)} \sum_{S\in\FM{\dtop}(\Ih(T))} \sum_{F\in\FM{k}(\Ih(S))}\tilde\lambda_F^2 h_S^{\dtop - 2k} \\
      \overset{\eqref{eq:local.bound.b}}&\gtrsim \sum_{T\in\FM{\dtop}(\Mh)} h_T^{\dtop-2k}\sum_{f\in\FM{k}(\Mh(T))}\lambda_f^2,
    \end{aligned}
  \end{equation}
  where, in the last passage, we have used the mesh regularity condition to write $h_S \simeq h_T$ for all $S\in\FM{\dtop}(\Ih(T))$.
  Similarly, we have
  \begin{equation}\label{eq:norm.psi_h}
    \begin{aligned}
      \norm{L^2 \Lambda^k(\Omega)}{\tilde\xi_h}^2
      \overset{\eqref{eq:def.psi_h}}&= \sum_{T\in\FM{\dtop}(\Mh)} \sum_{S\in\FM{\dtop}(\Ih(T))} \Norm{S}{\sum_{F\in\FM{k+1}(\Ih(S))}\tilde\xi_F \phi^{k+1}_{S,F}}^2 \\
      \overset{\text{\eqref{eq:Whitney.topo}}}&\simeq \sum_{T\in\FM{\dtop}(\Mh)} \sum_{S\in\FM{\dtop}(\Ih(T))} \sum_{F\in\FM{k+1}(\Ih(S))}\tilde\xi_F^2 \norm{S}{\phi^{k+1}_{S,F}}^2 \\
      \overset{\text{\eqref{eq:Whitney.norm}}}&\simeq \sum_{T\in\FM{\dtop}(\Mh)} \sum_{S\in\FM{\dtop}(\Ih(T))} \sum_{F\in\FM{k+1}(\Ih(S))}\tilde\xi_F^2 h_S^{\dtop - 2k-2} \\
      \overset{\eqref{eq:local.bound}}&\lesssim \sum_{T\in\FM{\dtop}(\Mh)} h_T^{\dtop-2k-2}\sum_{f\in\FM{k+1}(\Mh(T))}\xi_f^2,
    \end{aligned}
  \end{equation}
  where the conclusion follows from \eqref{eq:local.bound} applied to $\xi$, after noticing that $\cobd{k+1} \xi = \cobd{k+1} \cobd{k} \theta = 0$.

  The estimate \eqref{eq:topological.bound} follows combining \eqref{eq:prop.psi.phi}, \eqref{eq:norm.phi_h} and \eqref{eq:norm.psi_h}.
\end{proof}

%------------------------------------------------------------------------------%

\section{Proof of the main result}\label{sec:proof.poincare}

We are now ready to prove Theorem~\ref{thm:Poincare}.

\begin{proof}
  Let $\ul\omega_h\in \uH{k}{h}$ and set $\ul\sigma_h \coloneq \ul\DIFF^k_{r,h}\ul\omega_h$.
  We will construct $\ul\tau_h \in \uH{k}{h}$ such that $\ul\DIFF^k_{r,h}\ul\tau_h = \ul\sigma_h$
  and $\opn{h}{\ul\tau_h} \lesssim \opn{h}{\ul\sigma_h}$.
  \medskip\\
  \emph{(i) Construction of $\ul{\tau}_h$ on $k$-cells.}
  For $f\in\FM{k+1}(\Mh)$, we notice that
  \[
  \int_f \sigma_f\overset{\eqref{eq:def.ulDIFF}}=\int_f \star^{-1}\ltproj{0}{f}\star \DIFF^k_{r,f} \ul{\omega}_f=\int_f \DIFF_{r,f}^k\ul\omega_f,
  \]
  where the last equality follows from $\PLtrim{r}{0}(f)=\PL{r}{0}(f)$ together with $\star \DIFF^k_{r,f} \ul{\omega}_f\in\PL{r}{0}(f)$.
  The discrete Stokes formula \eqref{eq:def.d.int} with $\mu_f = 1\in\PL{r}{0}(f)$ and the definition \eqref{eq:def.P.d=k} then yield
  \[
  \int_f \sigma_f= \sum_{f'\in\FM{k}(\Mh(f))} \orffp \int_{f'} \omega_{f'}.
  \]
  This shows that the $(k+1)$-cochain $\xi \coloneq (\int_f \sigma_f)_{f \in \FM{k+1}(\Mh)} \in \cochain{k+1}(\Mh)$ is a coboundary (precisely, the coboundary of the $k$-cochain with coefficients $(\int_{f'}\omega_{f'})_{f' \in \FM{k}(\Mh)}$). Hence, by Lemma \ref{lemma:topological.bound} there exists a $k$-cochain $\lambda = (\lambda_{f'})_{f'\in\FM{k}(\Mh)}$ satisfying \eqref{eq:topological.bound} and such that $\cobd{k}\lambda=\xi$, which translates into: for all $f\in\FM{k+1}(\Mh)$, 
  \begin{equation} \label{eq:Pc.def0}
    \int_f \sigma_{f} =\sum_{f'\in\FM{k}(\Mh(\partial f))}\orffp\int_{f'}\tau_{f'},
  \end{equation}
  where we have defined
  \begin{equation}\label{eq:choice.star.tauf.0}
    \tau_{f'} \coloneq \star^{-1}\frac{\lambda_{f'}}{\vert f' \vert}.
  \end{equation}
  \medskip\\
  \emph{(ii) Construction of $\ul{\tau}_h$ on $d$-cells for $d \in \{ k+1, \dots, n\}$.}
  The construction is done by induction on $d$ (noticing that $d=k$ has already been done), by following the ideas in \cite[Lemma 27]{Bonaldi.Di-Pietro.ea:25}.
  For $d \in \{k+1, \dots, \dtop\}$ and $f\in\FM{d}(\Mh)$, assuming that $(\ul\tau_{f'})_{f'\in\FM{d-1}(f)}$ have been constructed,
  we define
  \begin{equation}\label{eq:choice.star.tauf}
    \tau_f \in\star^{-1}\DIFF\PL{r}{d - k - 1}(f)\subset\star^{-1}\PLtrim{r}{d-k}(f)
  \end{equation}
  such that, for all $\mu_f \in \KOSZUL\PL{r-1}{d-k}(f)$,
  \begin{equation} \label{eq:Pc.def1}
    (-1)^{k+1} \langle\star\tau_f, \DIFF\mu_f\rangle_f
    =  \langle\star\sigma_f, \mu_f\rangle_f
    -  \langle\star P^k_{r,\pf} \ul\tau_\pf, \tr_\pf \mu_f \rangle_\pf.
  \end{equation}
  The existence and uniqueness of $\tau_f$ follows from the fact that $\DIFF\st\KOSZUL\PL{r-1}{d-k}(f)\to\DIFF\PL{r}{d-k-1}(f)$ is an isomorphism.
  \medskip\\
  \emph{(iii) Proof that  $\ul\DIFF^k_{r,h}\ul\tau_h = \ul\sigma_h$.}
  Let $f\in\FM{d}(\Mh)$ for some $d \in \{ k+1, \dots, n\}$.
  Plugging \eqref{eq:Pc.def1} into the definition \eqref{eq:def.d} of the discrete exterior derivative for $\ul\tau_f$ we have, for all $\mu_f \in \KOSZUL\PL{r-1}{d-k}(f) \subset \PL{r}{d -k- 1}(f)$,
  \begin{equation}\label{eq:Pc.c.1}
    \langle \star \DIFF^k_{r,f} \ul\tau_f , \mu_f\rangle_f = \langle\star\sigma_f, \mu_f \rangle_f.
  \end{equation}

  We now prove that $\ul\DIFF^k_{r,f}\ul\tau_f = \ul\sigma_f$ by induction on the dimension $d$ of $f$.
  Let us assume first that $d=k+1$ and apply the definition \eqref{eq:def.d} of $\DIFF^k_{r,f} \ul\tau_f$ with a test function $\alpha_f\in\PL{0}{0}(f)$ to get
  \begin{equation}\label{eq:Pc.c.1.1}
    \langle \star \DIFF^k_{r,f} \ul\tau_f , \alpha_f\rangle_f =\langle \star \tau_{\pf} , \alpha_f\rangle_{\pf}\overset{\eqref{eq:Pc.def0}}=\langle\star \sigma_f,\alpha_f\rangle_f.
  \end{equation}
  Adding together \eqref{eq:Pc.c.1} and \eqref{eq:Pc.c.1.1} shows that the projections of $\star \DIFF^k_{r,f} \ul\tau_f$ and $\star \sigma_f$ on $\KOSZUL\PL{r-1}{1}(f)+\PL{0}{0}(f)$ coincide.
  Since this space is $\PL{r}{0}(f)$ (cf. \cite[Equation (2.7a)]{Bonaldi.Di-Pietro.ea:25}), which contains both $\star \DIFF^k_{r,f} \ul\tau_f$ and $\star \sigma_f$, this shows that $\DIFF^k_{r,f} \ul\tau_f=\sigma_f$, that is, that $\ul\DIFF_{r,h}^k\ul\tau_h$ and $\ul\sigma_h$ have the same component on $f$.

  We now consider the case $d \in \{ k+2, \dots, n\}$.
  Since $\DIFF^{k+1}_{r,f}\ul\sigma_f = \DIFF^{k+1}_{r,f}\ul\DIFF_{r,f}^k\ul\omega_f= 0$ (see \cite[Eq.~(3.32)]{Bonaldi.Di-Pietro.ea:25}), we have,
  for all $\zeta_f \in \KOSZUL\PL{r-1}{d -k-1}(f)\subset\PLtrim{r}{d-k-2}(f)$,
  \begin{equation}\label{eq:Pc.c.0}
    \begin{aligned}
      0 \overset{\eqref{eq:def.d}}&=
      (-1)^{k+2} \langle \star \sigma_f , \DIFF \zeta_f\rangle_f
      + \langle\star P^{k+1}_{r,\pf} \ul{\sigma}_\pf , \tr_\pf \zeta_f \rangle_\pf
      \\
      \overset{\eqref{eq:projP}}&=
      (-1)^{k+2} \langle \star \sigma_f , \DIFF \zeta_f\rangle_f
      + \langle\star {\sigma}_\pf , \tr_\pf \zeta_f \rangle_\pf,
    \end{aligned}
  \end{equation}
  where we have used the fact that
  $\tr_{f'} \zeta_f \in \PLtrim{r}{d -k- 2}(f')$ for all $f' \in \FM{d-1}(f)$ (cf. \cite[Lemma~4]{Bonaldi.Di-Pietro.ea:25}) to use \eqref{eq:projP} (with $(k+1,\ul\sigma_{f'})$, for $f'\in\FM{d-1}(\Mh(\partial f))$, instead of $(k,\ul\omega_f)$).
  Since $\zeta_f\in \PLtrim{r+1}{d-k-2}(f)$, we can apply the link between discrete exterior derivatives on subcells \cite[Lemma 22]{Bonaldi.Di-Pietro.ea:25} to find
  \begin{equation} \label{eq:Pc.c.2}
    \langle \star \DIFF^k_{r,f} \ul\tau_f , \DIFF \zeta_f \rangle_f
    = (-1)^{k+1} \langle \star \underbrace{\DIFF^k_{r,\pf} \ul\tau_\pf}_{\sigma_{\pf}}, \tr_\pf \zeta_f \rangle_\pf
    \overset{\eqref{eq:Pc.c.0}}=\langle \star \sigma_f , \DIFF \zeta_f\rangle_f,
  \end{equation}
  where the equality $\DIFF^k_{r,\pf} \ul\tau_\pf=\sigma_{\pf}$ comes from the induction hypothesis and the link \cite[Eq.~(3.31)]{Bonaldi.Di-Pietro.ea:25} between the potential reconstruction and the local discrete exterior derivative.
  Since $\PLtrim{r}{d-k-1}(f)=\DIFF(\KOSZUL\PL{r-1}{d -k - 1}(f))+ \KOSZUL\PL{r-1}{d -k}(f)$ (see, e.g., \cite[Eq.~(2.16)]{Bonaldi.Di-Pietro.ea:25}
  and recall that $d\ge k+2$ here), adding together \eqref{eq:Pc.c.1} and \eqref{eq:Pc.c.2}
  shows that $\ltproj{d-k-1}{f}\star\DIFF^k_{r,f} \ul\tau_f=\ltproj{d-k-1}{f}\star\sigma_f$,
  i.e., that $\ul\DIFF^k_{r,f}\ul\tau_h$ and $\ul\sigma_h$ have the same component on $f$ (see \eqref{eq:def.ulDIFF}).
  \medskip\\
  \emph{(iv) Proof that $\opn{h}{\ul\tau_h} \lesssim \opn{h}{\ul\sigma_h}$}.
  Let $f\in\FM{d}(\Mh)$ with $d \in \{ k+1, \dots, n\}$. By \eqref{eq:choice.star.tauf} we have $\star\tau_f\in\DIFF\PL{r}{d-k-1}(f)=\DIFF\KOSZUL\PL{r-1}{d-k}(f)$ (see \cite[Eq.~(2.8)]{Bonaldi.Di-Pietro.ea:25}). We can therefore use Lemma \ref{lemma:local.Poincare} to find $\mu_f \in \KOSZUL\PL{r-1}{d-k}(f)$ such that
  $\DIFF \mu_f = \star \tau_f$ and
  \begin{equation}\label{eq:mu.preimage.tau}
    \norm{f}{\mu_f}\lesssim h_f \norm{f}{\star\tau_f}.
  \end{equation}
  Using this $\mu_f$ in \eqref{eq:Pc.def1} and applying Cauchy--Schwarz inequalities, we have
  \begin{equation*}%\label{eq:Pc.b.1}
    \begin{aligned}
      \norm{f}{\star\tau_f}^2
      \lesssim{}& \norm{f}{\star\sigma_f}\norm{f}{\mu_f}
      + \sum_{f'\in \FM{d-1}(\Mh(\partial f))} \norm{f'}{\star P^k_{r,f'} \ul\tau_{f'}}\norm{f'}{\mu_f}\\
      \overset{\eqref{eq:trace.discrete}}{\lesssim}{}& \norm{f}{\star\sigma_f}\norm{f}{\mu_f}
      + h_f^{\frac12}\sum_{f'\in\FM{d-1}(\Mh(\partial f))} \norm{f'}{\star P^k_{r,f'} \ul\tau_{f'}}~h_f^{-1}\norm{f}{\mu_f}\\
      \overset{\eqref{eq:mu.preimage.tau}}{\lesssim}{}& \left( h_f \norm{f}{\star\sigma_f}
      + h_f^{\frac12}\sum_{f'\in\FM{d-1}(\Mh(\partial f))} \norm{f'}{\star P^k_{r,f'} \ul\tau_{f'}} \right) \norm{f}{\star\tau_f}.
    \end{aligned}
  \end{equation*}
  Using \eqref{eq:Pb.m1} (with $f'$ instead of $f$) together the fact that $\star$ is an isometry, squaring the result and using $\card(\FM{d-1}(\Mh(\partial f)))\lesssim 1$ by mesh regularity assumption, we infer
  \begin{equation}\label{eq:Pc.b.2}
    \norm{f}{\tau_f}^2 \lesssim h_f^2 \norm{f}{\sigma_f}^2 + h_f\sum_{f'\in\FM{d-1}(\Mh(\partial f))} \opn{f'}{\ul\tau_{f'}}^2.
  \end{equation}

  Let us now use this relation to show by induction on $d\in \{k+1, \dots, n\}$ that, for all $f\in\FM{d}(\Mh)$,
  \begin{equation}\label{eq:Pc.b.recopn}
    \opn{f}{\ul{\tau}_{f}}^2
    \lesssim h_f^2  \opn{f}{\ul{\sigma}_f}^2
    + h_f^{d-k}\sum_{f'\in\FM{k}(\Mh(f))}\norm{f'}{\tau_{f'}}^2.
  \end{equation}
  For $d=k+1$, this relation is a straightforward consequence of \eqref{eq:Pc.b.2}, since $\opn{f'}{\ul\tau_{f'}}=\norm{f'}{\tau_{f'}}$ for all $f'\in\FM{k}(\Mh)$ (see \eqref{eq:opn.f.k}).
  Assuming now that \eqref{eq:Pc.b.recopn} holds for some $d\in\{k+1, \dots, n-1\}$, we establish it for $(d+1)$-cells.
  By definition \eqref{eq:opn.f} of the local triple norm we have,
  for all $f\in\FM{d+1}(\Mh)$,
  \begin{equation*}
    \begin{aligned}
      \opn{f}{\ul{\tau}_f}^2
      &= \norm{f}{\tau_f}^2 + h_f \sum_{f'\in\FM{d}(\Mh(\partial f))}\opn{f'}{\ul{\tau}_{f'}}^2\\
      \overset{\eqref{eq:Pc.b.2}}&\lesssim
      h_f^2\norm{f}{\sigma_f}^2
      + h_f  \sum_{f'\in\FM{d}(\Mh(\partial f))}\opn{f'}{\ul{\tau}_{f'}}^2\\
      &\lesssim
      h_f^2\norm{f}{\sigma_f}^2
      + h_f^3  \sum_{f'\in\FM{d}(\Mh(\partial f))}\opn{f'}{\ul{\sigma}_{f'}}^2
      + h_f^{d+1-k}  \sum_{f''\in\FM{k}(\Mh(f))}\norm{f''}{\tau_{f''}}^2 \\
      &\lesssim
      h_f^2  \opn{f}{\ul{\sigma}_f}^2
      + h_f^{d+1-k}  \sum_{f'\in\FM{k}(\Mh(f))}\norm{f'}{\tau_{f'}}^2,
    \end{aligned}
  \end{equation*}
  where the second inequality follows from \eqref{eq:Pc.b.recopn} applied to each $f'\in\FM{d}(\Mh(\partial f))$,
  together with $h_{f'}\le h_f$ and $\sum_{f'\in\FM{d}(\Mh(\partial f))}\sum_{f''\in\FM{k}(\Mh(f'))}\bullet\lesssim \sum_{f''\in\FM{k}(\Mh(f))}\bullet$
  (since each the number of occurrences of each $f''$ in the sum on the left is uniformly bounded by mesh regularity assumption).
  To conclude we have used the definition of $\opn{f}{\ul{\sigma}_f}$. This concludes the induction and therefore the proof of \eqref{eq:Pc.b.recopn}.

  Summing \eqref{eq:Pc.b.recopn} over $n$-cells, we get
  \begin{equation} \label{eq:Pc.b.3}
    \opn{h}{\ul{\tau}_h}^2
    \lesssim \sum_{f\in\FM{\dtop}(\Mh)} h_f^2 \opn{f}{\ul\sigma_f}^2
    + \sum_{f\in\FM{\dtop}(\Mh)} h_f^{\dtop - k} \sum_{f'\in\FM{k}(\Mh(f))} \norm{f'}{\tau_{f'}}^2.
  \end{equation}
  The definition \eqref{eq:choice.star.tauf.0} of $\ul\tau_h$ on $k$-cells and the fact that $\star$ is an isometry yields, for all $f'\in\FM{k}(\Mh)$,
  \[
  \norm{f'}{\tau_{f'}}^2=\frac{\lambda_{f'}^2}{|f'|}\lesssim h_{f'}^{-k}\lambda_{f'}^2,
  \]
  where the inequality is obtained by invoking the mesh regularity assumption to infer $|f'| \simeq h^k_{f'}$.
  Plugging this relation into \eqref{eq:Pc.b.3} and recalling that $(\lambda_{f'})_{f'\in\FM{k}(\Mh)}$ satisfies \eqref{eq:topological.bound} with $(\xi_{f'})_{f'\in\FM{k+1}(\Mh)}=(\int_{f'}\sigma_{f'})_{f'\in\FM{k+1}(\Mh)}$, we obtain
  \begin{equation*}
    \begin{aligned}
      \opn{h}{\ul\tau_h}^2
      \lesssim{}& \sum_{f\in\FM{\dtop}(\Mh)} h_f^2 \opn{f}{\ul\sigma_f}^2
      + \sum_{f\in\FM{\dtop}(\Mh)} h_f^{\dtop - 2 k} \sum_{f'\in\FM{k}(\Mh(f))} \lambda_{f'}^2 \\
      \lesssim{}& \sum_{f\in\FM{\dtop}(\Mh)} h_f^2 \opn{f}{\ul\sigma_f}^2
      + \sum_{f\in\FM{\dtop}(\Mh)} h_f^{\dtop - 2 (k + 1)} \sum_{f'\in\FM{k+1}(\Mh(f))} \norm{f'}{\sigma_{f'}}^2 \vert f'\vert \\
      \lesssim{}& \sum_{f\in\FM{\dtop}(\Mh)} h_f^2 \opn{f}{\ul\sigma_f}^2
      + \sum_{f\in\FM{\dtop}(\Mh)} h_f^{\dtop - (k + 1)} \sum_{f'\in\FM{k+1}(\Mh(f))} \norm{f'}{\sigma_{f'}}^2 \\
      \lesssim{}& \opn{h}{\ul\sigma_h}^2
    \end{aligned}
  \end{equation*}
  where we used a Cauchy--Schwarz inequality on the second line, and $\vert f' \vert \simeq h_f^{k+1}$ (by mesh regularity)
  for all $f'\in\FM{k+1}(\Mh(f))$ in the third line.
  The conclusion follows from the definition \eqref{eq:def.hnorm} of $\opn{h}{\ul\sigma_h}$ using the equivalent formulation \eqref{eq:opn.f.expl} of the local norm.
\end{proof}

%------------------------------------------------------------------------------%

\section*{Acknowledgements}

Daniele Di Pietro, J\'er\^{o}me Droniou, and Marien Hanot acknowledge the funding of the European Union via the ERC Synergy, NEMESIS, project number 101115663.
Silvano Pitassi acknowledges the funding of the European Union via the MSCA EffECT, project number 101146324.
Views and opinions expressed are however those of the authors only and do not necessarily reflect those of the European Union or the European Research Council Executive Agency.
Neither the European Union nor the granting authority can be held responsible for them.

%------------------------------------------------------------------------------%

\appendix

\section{Results on local polynomial spaces}\label{app:polynomial.spaces}

Some of the results established in this appendix on polynomial forms have already been proved in \cite{Di-Pietro.Droniou:20,Di-Pietro.Droniou:23*1} for polynomial functions and in the vector calculus setting.

For $d \in \{1, \dots, \dtop\}$, we denote by $B_d(x,r)$ the ball in $\Real^d$ of radius $r$ centered at $x$,
and set $B_d \coloneq B_d(0,1)$ for the sake of brevity.
For any $f\in\FM{d}(\Mh)$, we define the function $\psi_f \st x \to  x_f+h_fx$.
Denoting by $\rho$ the mesh regularity parameter, we have
$\psi_f(B_d(0,\rho))=B_d(x_f, \rho h_f) \subset f \subset \psi_f(B_d)=B_d(x_f,h_f)$.
We recall that, if $g$ is an $l$-form, $\psi_f^\star g$ denotes the pullback of $g$ by $\psi_f$, that is,
the $l$-form such that $(\psi^\star g)_x(v_1,\ldots,v_l)=g_{\psi_f(x)}(D\psi_f(x)v_1,\ldots,D\psi_f(x)v_l)$.

\begin{lemma}[Norm equivalence]
  For any $g\in L^2\Lambda^l(\psi_f(B_d))$,
  it holds
  \begin{equation}
    \norm{B_d(0,\rho)}{\psi_f^\star g} \leq h_f^{l-\frac{d}{2}}\norm{f}{g} \leq \norm{B_d}{\psi_f^\star g}.
    \label{eq:equiv.B}
  \end{equation}
\end{lemma}

\begin{proof}
  Let $g\in L^2\Lambda^l(\psi_f(B_d))$.
  By definition of the pullback, and since $D\psi_f = h_f\mathrm{Id}$, we have, by linearity of $g_{\psi_f(x)}$, $(\psi_f^\star) g_x = h_f^l g_{\psi_f(x)}$.
  Moreover, $\vert D\psi_f \vert = h_f^d$.
  Therefore, using the change of variable formula and denoting by $\norm{\Alt^l(\Real^d)}{{\cdot}}$ the canonical norm (induced by the canonical inner product) on the space $\Alt^l(\Real^d)$  of alternating $l$-linear forms on $\Real^d$, we have
  \begin{align*}
    \norm{B_d}{\psi_f^\star g}^2 &= \int_{B_d} \norm{\Alt^l(\Real^d)}{(\psi_f^\star g)_x}^2 \, \DIFF x
    = h_f^{-d} \int_{B_d} h_f^{2l} \norm{\Alt^l(\Real^d)}{g_{\psi_f(x)}}^2 \vert D\psi_f \vert \, \DIFF x\\
    &= h_f^{2l-d} \int_{\psi_f(B_d)} \norm{\Alt^l(\Real^d)}{g_x}^2 \, \DIFF x
    = h_f^{2l-d} \norm{\psi_f(B_d)}{g}^2 = h_f^{2l-d} \norm{B_d(x_f,h_f)}{g}^2.
  \end{align*}
  Likewise, we have $\norm{B_d(0,\rho)}{\psi_f^\star g}^2 = h_f^{2l-d}\norm{B_d(x_f,\rho h_f)}{g}^2$.
  We conclude by using $B_d(x_f, \rho h_f) \subset f \subset B_d(x_f,h_f)$ to write
  \[
  \norm{B_d(0,\rho)}{\psi_f^\star g}^2 = h_f^{2l-d}\norm{B_d(x_f,\rho h_f)}{g}^2 \leq h_f^{2l-d}\norm{f}{g}^2 \leq h_f^{2l-d} \norm{B_d(x_f,h_f)}{g}^2 =\norm{B_d}{\psi_f^\star g}^2.\qedhere
  \]
\end{proof}

If $P \in \PL{r}{l}(\Real^d)$, then the equivalence of norms in finite dimension ensures that
$\norm{B_d(0,\rho)}{P} \simeq \norm{B_d}{P}$, with hidden constants depending only on the space involved, that is, on $r$, $l$ and $d$.
Taking $\omega\in\PL{r}{l}(\Real^d)$ and applying this remark to $P = \psi_f^\star \omega\in \PL{r}{l}(\Real^d)$ (since $\psi_f$ is linear),
we see that \eqref{eq:equiv.B} yields
\begin{equation}
  h_f^{l-\frac{d}{2}} \norm{f}{\omega} \simeq \norm{B_d}{\psi_f^\star \omega}\qquad\forall\omega\in\PL{r}{l}(\Real^d).
  \label{eq:equiv.P}
\end{equation}
In the following lemmas, when $A$ is a mapping between spaces of polynomial forms on some $f\in\FM{d}(\Mh)$, we denote by $\vvvert A\vvvert$ the norm of $A$ induced by the $L^2(f)$-norms on its domain and co-domains.

\begin{lemma}[Discrete inequalities for $\DIFF$ and $\KOSZUL$]\label{lemma:local.invPoincare}
  For any $f\in\FM{d}(\Mh)$, the differential $\DIFF\st\PL{r}{d-k-1}(f)\to\PL{r-1}{d-k}(f)$ and the Koszul operator $\KOSZUL\st\PL{r}{d-k+1}(f)\to\PL{r+1}{d-k}(f)$ satisfy
  \begin{align}\label{eq:local.d.bound}
    \vvvert \DIFF \vvvert_f \lesssim{}& h_f^{-1},\\
    \label{eq:local.k.bound}
    \vvvert \KOSZUL \vvvert_f \lesssim{}& h_f.
  \end{align}
\end{lemma}

\begin{proof}
  The proof of \eqref{eq:local.d.bound} hinges on the fact that the exterior derivative commutes with pullbacks.
  The operator $\DIFF$ is continuous on the finite dimensional space $\PL{r}{d-k-1}(\Real^d)$, with a continuity constant that only depends on $r$, $k$ and $d$. Hence,
  \begin{equation}\label{eq:cont.d0}
    \norm{B_d}{\DIFF\mu} \lesssim \norm{B_d}{\mu} \qquad \forall \mu\in \PL{r}{d-k-1}(\Real^d).
  \end{equation}
  We then write, for all $\omega\in \PL{r}{d-k-1}(f)$ and since $\DIFF\omega\in \PL{r-1}{d-k}(f)$,
  \begin{multline}\label{eq:proof.cont.d}
    \norm{f}{\DIFF\omega}\overset{\eqref{eq:equiv.P}}\simeq h_f^{\frac{d}{2}-(d-k)}\norm{B_d}{\psi_f^\star (\DIFF\omega)}
    =h_f^{k-\frac{d}{2}}\norm{B_d}{\DIFF(\psi_f^\star\omega)}\\
    \overset{\eqref{eq:cont.d0}}\lesssim
    h_f^{k-\frac{d}{2}}\norm{B_d}{\psi_f^\star\omega}\overset{\eqref{eq:equiv.P}}\simeq h_f^{k-\frac{d}{2}}h_f^{d-k-1-\frac{d}{2}}\norm{f}{\omega},
  \end{multline}
  which concludes the proof of \eqref{eq:local.d.bound} since $k-\frac{d}{2}+d-k-1-\frac{d}{2}=-1$.

  The same approach can be applied to the Koszul operator, once we notice that it has a similar
  commutation property with the considered pullbacks.
  Specifically, recalling that $\KOSZUL = i_{x-x_f}$ is the Koszul operator on $f$,
  and defining the Koszul operator at 0 by $\KOSZUL_0 \coloneq i_{x}$, the relation
  \begin{equation}\label{eq:commut.koszul}
    (\psi_f^\star \KOSZUL)_{x} = \left(i_{h_f^{-1} \left( x_f+h_fx - x_f \right)} \psi_f^\star\right)_{x} = \left( \KOSZUL_0 \psi_f^\star \right)_{x}
  \end{equation}
  comes from the generic formula $\phi^\star \left( i_{Y}\mu \right) = i_{\phi_\star Y}\phi^\star\mu$, where $(\phi_\star Y)(x) \coloneq (D\phi(x))^{-1}Y(\phi(x))$.
  Equipped with this commutation and using the continuity of $\KOSZUL_0$ on $\PL{r}{d-k+1}(\Real^d)$, namely
  \begin{equation}\label{eq:cont.koszul0}
    \norm{B_d}{\KOSZUL_0\mu} \lesssim \norm{B_d}{\mu} \qquad \forall \mu\in \PL{r}{d-k+1}(\Real^d),
  \end{equation}
  we can reproduce the same arguments as in \eqref{eq:proof.cont.d} to prove \eqref{eq:local.k.bound}, the change
  of scaling from $h_f^{-1}$ to $h_f$ coming from the fact that, in \eqref{eq:proof.cont.d}, the form degree of $\DIFF\omega$ is one more than
  that of $\omega$, while it is one less for $\KOSZUL\omega$.
\end{proof}

\begin{lemma}[Local discrete Poincaré inequality]\label{lemma:local.Poincare}
  For any $f\in\FM{d}(\Mh)$, the inverse of the differential $\DIFF^{-1}\st\DIFF\KOSZUL\PL{r}{k}(f)\to\KOSZUL\PL{r}{k}(f)$ and the inverse of the Koszul operator $\KOSZUL^{-1}\st\KOSZUL\DIFF\PL{r}{k}(f)\to\DIFF\PL{r}{k}(f)$  satisfy
  \begin{align}\label{eq:local.Poincare.d}
    \vvvert \DIFF^{-1} \vvvert_f \lesssim{}& h_f,\\
    \label{eq:local.Poincare.k}
    \vvvert \KOSZUL^{-1} \vvvert_f \lesssim{}& h_f^{-1} .
  \end{align}
\end{lemma}

\begin{proof}
  The proof uses the same arguments as the proof of Lemma \ref{lemma:local.invPoincare},
  using this time the continuity of $\DIFF^{-1} \st \DIFF\KOSZUL_0\PL{r}{k}(B_d) \to \KOSZUL_0\PL{r}{k}(B_d)$,
  and of $\KOSZUL_0^{-1} \st \KOSZUL_0\DIFF\PL{r}{k}(B_d)\to \DIFF\PL{r}{k}(B_f)$ (with norms only depending on $r$, $k$ and $d$).
  Let $\mu \in \DIFF\KOSZUL\PL{r}{k}(f)$ and set $\omega\in \PL{r}{k}(f)$ such that $\mu = \DIFF\KOSZUL \omega$. Then $\DIFF^{-1}\mu=\KOSZUL\omega$ and
  \begin{multline*}
    h_f^{k-1-\frac{d}{2}}\norm{f}{\DIFF^{-1} \mu} \overset{\eqref{eq:equiv.P}}\simeq
    \norm{B_d}{\psi_f^\star \KOSZUL \omega}
    \overset{\eqref{eq:commut.koszul}}= \norm{B_d}{\DIFF^{-1}\DIFF\KOSZUL_0\psi_f^\star \omega}
    \\
    \lesssim \norm{B_d}{\DIFF\KOSZUL_0\psi_f^\star \omega}
    \overset{\eqref{eq:cont.d0},\eqref{eq:cont.koszul0}}\lesssim
    \norm{B_d}{\psi_f^\star \omega}
    \overset{\eqref{eq:equiv.P}}\simeq h_f^{k-\frac{d}{2}}\norm{f}{\omega}.
  \end{multline*}
  Simplifying by $h_f^{k-1-\frac{d}{2}}$ concludes the proof of \eqref{eq:local.Poincare.d}.
  The proof of \eqref{eq:local.Poincare.k} follows from similar arguments.
\end{proof}

\begin{lemma}[Topological decomposition]\label{lemma:topological.decomp}
  The decomposition
  \[
  \PL{r}{k}(f) = \DIFF\PL{r+1}{k-1}(f)\oplus \KOSZUL\PL{r-1}{k+1}(f)
  \]
  is topological:
  For all $(\mu_f,\nu_f) \in \DIFF\PL{r+1}{k-1}(f)\times \KOSZUL\PL{r-1}{k+1}(f)$,
  \[
  \norm{f}{\mu_f} + \norm{f}{\nu_f} \simeq \norm{f}{\mu_f + \nu_f} .
  \]
\end{lemma}

\begin{proof}
  The inequality $\gtrsim$ directly follows from a triangle inequality, so we focus on $\lesssim$.
  We have
  \[
  \norm{f}{\mu_f} = \norm{f}{\KOSZUL^{-1}\KOSZUL\mu_f}
  \overset{\eqref{eq:local.Poincare.k}}\lesssim h_f^{-1}\norm{f}{\KOSZUL\mu_f}
  = h_f^{-1}\norm{f}{\KOSZUL(\mu_f + \nu_f)}
  \overset{\eqref{eq:local.k.bound}}\lesssim\norm{f}{\mu_f+\nu_f},
  \]
  where we have used the fact that $\KOSZUL \nu_f \in\KOSZUL^2\PL{r-1}{k+1}(f)= \{0\}$ in the third passage.
  We also have
  \[
  \norm{f}{\nu_f} = \norm{f}{\DIFF^{-1}\DIFF\nu_f}
  \overset{\eqref{eq:local.Poincare.d}}\lesssim h_f\norm{f}{\DIFF\nu_f}
  = h_f\norm{f}{\DIFF(\nu_f + \mu_f)}
  \overset{\eqref{eq:local.d.bound}}\lesssim \norm{f}{\nu_f+\mu_f},
  \]
  where we have used $\DIFF \mu_f\in \DIFF^2\PL{r+1}{k-1}(f)= \{0\}$ in the third passage. Combining these two relations
  yields the required estimate.
\end{proof}

\begin{lemma}[Discrete trace inequality]\label{lemma:trace.discrete}
  For all $f\in\FM{d}(\Mh)$
  and all $\omega_f\in\PL{r}{k}(f)$ it holds
  \begin{equation}\label{eq:trace.discrete}
    \norm{\pf}{\omega_f} \lesssim h_f^{-\frac12} \norm{f}{\omega_f}.
  \end{equation}
\end{lemma}

\begin{proof}
  We introduce the notations
  \[
  \mathfrak{A}_l^d \coloneq \left\lbrace \alpha \in \mathbb{N}^d \st \sum_{i = 1}^d \alpha_i = l \right\rbrace, \quad
  \mathfrak{B}_k^d \coloneq \left\lbrace \beta \in \lbrace 0, 1 \rbrace^d \st \sum_{i = 1}^d \beta_i = k \right\rbrace,
  \]
  for the sets of multi-indices corresponding to the set of monomials of degree $l$, and to a basis of $k$-forms.
  Specifically, letting $y^{\alpha}\coloneq y_1^{\alpha_1}\cdots y_d^{\alpha_d}$ and, if $i_1<\ldots<i_k$ are the indices such that $\beta_i=1$ if and only if $i\in\{i_1,\ldots,i_k\}$, $\DIFF x^\beta\coloneq\DIFF x^{i_1}\wedge\cdots\wedge \DIFF x^{i_k}$,
  the set of monomials $\left\{ (x-x_f)^\alpha\DIFF x^{\beta} \right\}_{l\in\left\{ 0,\dots,r\right\},\alpha\in \mathfrak{A}_l^d,\beta\in \mathfrak{B}_k^d}$ forms a basis of $\PL{r}{k}(\Real^d)$. Any $\omega_f \in \PL{r}{k}(f)$ can thus be evaluated at a point $x$ according to the following formula:
  \[
  (\omega_f)_x = \sum_{l=0}^{r}\sum_{\alpha\in \mathfrak{A}_l^d,\beta\in \mathfrak{B}_k^d} \lambda_{l,\alpha,\beta} (x-x_f)^{\alpha} \DIFF x^{\beta}
  \]
  and, using the equivalence of norms in finite dimension and the fact that $\left\{ x^\alpha\DIFF x^{\beta} \right\}_{l\in\left\{ 0,\dots,r\right\},\alpha\in \mathfrak{A}_l^d,\beta\in \mathfrak{B}_k^d}$ is a basis of $\PL{r}{k}(\Real^d)$, we can write
  \begin{equation}
    \begin{aligned}
      \norm{f}{\omega_f}^2 \overset{\eqref{eq:equiv.P}}&\simeq h_f^{-2k+d}\norm{B_d}{\psi_f^\star \omega_f}^2
      = h_f^{-2k+d}
      \left\Vert \sum_{l=0}^{r}\sum_{\alpha\in \mathfrak{A}_l^d,\beta\in \mathfrak{B}_k^d}
      \lambda_{l,\alpha,\beta} (h_f x)^{\alpha} h_f^k \DIFF x^{\beta}\right\Vert_{B_d}^2 \\
      &\simeq h_f^{-2k+d} \sum_{l=0}^{r}\sum_{\alpha\in \mathfrak{A}_l^d,\beta\in \mathfrak{B}_k^d}
      h_f^{2l + 2k} \lambda_{l,\alpha,\beta} ^2
      = h_f^d \sum_{l=0}^{r}\sum_{\alpha\in \mathfrak{A}_l^d,\beta\in \mathfrak{B}_k^d}
      h_f^{2l} \lambda_{l,\alpha,\beta} ^2 .
    \end{aligned}
    \label{eq:l2.omegaf}
  \end{equation}
  Let $f' \in \FM{d-1}(\Mh(f))$.
  We define $x_{f\!f'}$ as the orthogonal projection of $x_f$ onto the tangent plane to $f'$,
  so that $x_f = x_{f\!f'} + c n_{f'}$ for some $c\in\Real$ such that $|c|\le h_f$ and with $n_{f'}$ a unit normal vector to $f'$ in the space spanned by $f$.
  Assuming, without loss of generality, that $n_{f'}$ is along the $d$-th basis vector,
  we notice that, for any multi-index $\alpha = (\alpha_1, \dots, \alpha_d) \in \mathbb{N}^d$,
  $((x-x_f)^\alpha)_{\vert f'} = (x_{\vert f'} - x_{f\!f'})^{(\alpha_1, \dots, \alpha_{d-1})} c^{\alpha_d}$.
  Therefore, denoting by
  \[
  (\omega_f)_{\vert f'} = \sum_{l=0}^{r}\sum_{\alpha'\in \mathfrak{A}_l^{d-1},\beta'\in \mathfrak{B}_k^{d-1}} \lambda_{l,\alpha',\beta'}' (x_{\vert f'}-x_{f\!f'})^{\alpha'} \DIFF x^{\beta'}
  \]
  the trace of $\omega_f$ on $f'$, we have
  \[
  \lambda_{l,\alpha',\beta'}' = \sum_{l'=l}^r c^{l'-l} \lambda_{l',(\alpha',l'-l),\beta'}.
  \]
  Proceeding as in \eqref{eq:l2.omegaf}, we have
  \begin{equation*}
    \begin{aligned}
      \norm{f'}{(\omega_f)_{\vert f'}}^2 &\simeq
      h_{f'}^{d-1} \sum_{l=0}^{r}\sum_{\alpha'\in \mathfrak{A}_l^{d-1},\beta'\in \mathfrak{B}_k^{d-1}}h_{f'}^{2l} \lambda_{l,\alpha',\beta'}'^2\\
      &\lesssim
      h_{f'}^{d-1} \sum_{l=0}^{r}\sum_{\alpha'\in \mathfrak{A}_l^{d-1},\beta'\in \mathfrak{B}_k^{d-1}} \sum_{l'=l}^r h_{f'}^{2l}c^{2(l'-l)} \lambda_{l',(\alpha',l'-l),\beta'}^2\\
      \overset{\vert c\vert\leq h_f}&\simeq
      h_{f}^{d-1} \sum_{l=0}^{r}\sum_{\alpha'\in \mathfrak{A}_l^{d-1},\beta'\in \mathfrak{B}_k^{d-1}} \sum_{l'=l}^r h_{f}^{2l'}\lambda_{l',(\alpha',l'-l),\beta'}^2
      \overset{\eqref{eq:l2.omegaf}}\lesssim h_f^{-1} \norm{f}{\omega_f}^2,
    \end{aligned}
  \end{equation*}
  where, in the conclusion, we have used the mesh regularity assumption to write $h_{f'} \simeq h_f$.
\end{proof}

\section{Whitney forms}\label{section:whitney}

The proof of the Poincar\'e inequality in the lowest order case (Lemma \ref{lemma:topological.bound}) relies on the use of a conforming basis of polynomial forms on a simplicial mesh, given by the Whitney forms.
We recall here their definition and key properties.

Any simplex $T$ of dimension $\dtop$ is the convex hull of some vertices $v_0, \dots, v_{\dtop}$ not contained in a hyperplane.
We denote by $\lambda_0,\dots,\lambda_{\dtop}$ the barycentric coordinates associated to these vertices.
The exterior derivatives $\DIFF\lambda_0,\dots\DIFF\lambda_{\dtop}$ are constant $1$-forms spanning  the space of $1$-forms on $\mathbb{R}^n$.
They satisfy the relations
\begin{equation}\label{eq:bary.constraint}
  \text{%
    $\sum_{i=0}^{\dtop} \lambda_i = 1$\quad
    and\quad
    $\sum_{i=0}^{\dtop} \DIFF\lambda_i = 0$.
  }
\end{equation}

For any $k \in \{0, \dots, \dtop\}$, we denote by $\Sigma(k)$ the set of strictly increasing maps $\lbrace 0,\dots,k\rbrace\mapsto\lbrace 0,\dots,\dtop\rbrace$.
For $\sigma\in\Sigma(k)$, we denote the range of $\sigma$ by $[\sigma] \coloneq \lbrace \sigma(i) \rbrace_{0 \leq i \leq k}$,
and, for $p \notin [\sigma]$, we set $\epsilon(p,\sigma) \coloneq (-1)^{\card\left(\{q \in [\sigma] \;:\; p > q\}\right)}$.
For $p \in [\sigma]$, we denote by $\sigma - p$ the unique element of $\Sigma(k-1)$ whose range is $[\sigma]\setminus p$,
and conversely, for $p \notin [\sigma]$, $\sigma + p$ is the unique element of $\Sigma(k+1)$ with range $[\sigma]\cup\lbrace p\rbrace$.
There is a one-to-one correspondence between $\Sigma(k)$ and the $k$-skeleton $\FM{k}(\Mh(T))$ of $T$,
associating each $\sigma\in\Sigma(k)$ with the convex hull of $\{v_l\}_{l\in [\sigma]}$; so, depending on the context, the notation $\sigma$ may also be used to indicate the latter.
We also define the basic $k$-alternator
$\DIFF\lambda_\sigma \coloneq \DIFF\lambda_{\sigma(0)}\wedge\dots\wedge\DIFF\lambda_{\sigma(k)}$.

For all $k \in \{0, \dots, \dtop\}$, and all $\sigma\in\Sigma(k)$, we define the Whitney form $\phi_{T,\sigma}^k\in \PLtrim{1}{k}(T)$ associated to $\sigma$ by
\begin{equation}\label{eq:def.whitney}
  \phi_{T,\sigma}^k \coloneq \sum_{p\in[\sigma]} \epsilon(p,\sigma - p)\lambda_p\DIFF\lambda_{\sigma - p} .
\end{equation}
We recall the following result from \cite[Eq.~(III.8)]{Licht:17*1}:
For all $\sigma,\rho\in\Sigma(k)$,
\begin{equation}\label{eq:whitney.dualbasis}
  \int_\sigma \phi^k_{T,\rho} =
  \begin{cases}
    \frac{1}{\dtop!} & \text{if $\sigma = \rho$,}
    \\
    0 & \text{if $\sigma \neq \rho$}.
  \end{cases}
\end{equation}

\begin{lemma}
  For all $k \in \{0, \dots, \dtop\}$ and all $\sigma\in\Sigma(k)$, we have
  \begin{equation}\label{eq:Whitney.norm}
    \norm{T}{\phi^k_{T,\sigma}}^2 \simeq h_T^{\dtop - 2k} .
  \end{equation}
  Moreover, for any family $\{ \mu_\sigma \}_{\sigma\in\Sigma(k)}$ of real numbers,
  \begin{equation}
    \left\| \sum_{\sigma\in\Sigma(k)} \mu_\sigma \phi^k_{T,\sigma} \right\|_T^2
    \simeq \sum_{\sigma\in\Sigma(k)} \mu_\sigma^2 \norm{T}{\phi_{T,\sigma}^k}^2.
    \label{eq:Whitney.topo}
  \end{equation}
\end{lemma}

\begin{proof}
  Let $U \in \Real^{n\times n}$ be the matrix whose $i$-th column is $v_i - v_0$.
  The barycentric coordinates associated with $T$ of a point $x$ are given by the relation
  $x = U\lambda + v_0$.
  Let $\psi \st \Real^n \ni x \mapsto Ux + v_0 \in \Real^n$, so that
  $\psi^{-1} \st \Real^n \ni x \mapsto U^{-1}\left( x - v_0 \right) \in \Real^n$.
  We notice that, for $p\neq 0$, $\lambda_p = \left(\psi^{-1}\right)^\star x_p$ and
  $\DIFF\lambda_p = \left( \psi^{-1} \right)^\star \DIFF x^p$.
  For $p = 0$, these relations still hold defining $x_0$ to satisfy the linear dependency \eqref{eq:bary.constraint}.
  Taking the pullback of \eqref{eq:def.whitney} by $\psi$ and using the second relation, above, we infer%
  \[
  \psi^\star \phi^k_{T,\sigma} = \sum_{p\in[\sigma]} \epsilon(p,\sigma-p)x_p\DIFF x^{\sigma -p},
  \]
  which is the definition of the Whitney forms on the reference $n$-simplex $S_0$.
  The relation \eqref{eq:whitney.dualbasis} shows that the family $\{ \phi^k_{T,\sigma} \}_{\sigma\in\Sigma(k)}$ is linearly independent, since it shows that the integrals over the sub-simplices of $T$ form a dual basis to the Whitney forms.
  Therefore, we infer from the equivalence of norms in finite dimension that,
  for all $\mu_\sigma \in \Real$ and with a hidden constant depending only on $S_0$ (which, in turn, depends only on $n$) and $k$
  \begin{equation}
    \left\|\psi^\star\sum_{\sigma\in\Sigma(k)} \mu_\sigma \phi^k_{T,\sigma}\right\|_{S_0}^2
    \simeq \sum_{\sigma\in\Sigma(k)} \mu_\sigma^2 \norm{S_0}{\psi^\star\phi_{T,\sigma}^k}^2
    \quad\text{ and }\quad
    \norm{S_0}{\psi^\star \phi^k_{T,\sigma}}^2 \simeq 1.
    \label{eq:Whitney.norm.P1}
  \end{equation}
  We now want to compare, for $\mu\in\SPAN\{\phi^k_{T,\sigma}\}_{\sigma\in\Sigma(k)}$, the norms of $\psi^\star\mu$ and $\mu$.
  To this end, we need an estimate of the pointwise norm $\norm{\Alt^k}{(\psi^\star \mu)_x}$, which is equivalent to the operator norm of $(\psi^\star \mu)_x$ viewed as a multilinear form since $\Alt^k$ is a finite-dimensional space.
  In the following, we denote by triple bars these operator norms.
  For any vectors $Y_1,\dots,Y_k$, we have
  \begin{align*}
    \vert (\psi^\star \mu)_{x}(Y_1,\dots,Y_k)\vert
    = \vert(\phi^k_{T,\sigma})_{\psi(x)}(UY_1,\dots,UY_k)\vert
    &\leq \vvvert \mu_{\psi(x)}\vvvert \Vert UY_1\Vert \dots \Vert UY_k\Vert \\
    &\leq \vvvert \mu_{\psi(x)}\vvvert \vvvert U \vvvert^k \Vert Y_1\Vert \dots \Vert Y_k \Vert.
  \end{align*}
  Applying the same argument to $U^{-1}Y_1,\ldots,U^{-1}Y_k$, we infer
  \[
  \vvvert U^{-1}\vvvert^{-k} \vvvert \mu_{\psi(x)}\vvvert \leq \vvvert (\psi^\star \mu)_{x}\vvvert \leq \vvvert U \vvvert^k \vvvert \mu_{\psi(x)}\vvvert.
  \]
  The shape regularity assumption on the simplex $T$ ensures that
  \[
  \vvvert U^{-1} \vvvert^{-1} \simeq \vvvert U \vvvert \simeq h_T.
  \]
  Therefore,
  $\norm{\Alt^k}{(\psi^\star \mu)_x}^2
  \simeq h_T^{2k} \norm{\Alt^k}{\mu_{\psi(x)}}^2$,
  and a change of variable yields
  \begin{equation}
    \begin{aligned}
      \norm{S_0}{\psi^\star\mu}^2
      &= \int_{S_0} \norm{\Alt^k}{(\psi^\star \mu)_x}^2\DIFF x \\
      &\simeq \int_{S_0} h_T^{2k} \norm{\Alt^l}{\mu_{\psi(x)}}^2 h_T^{-\dtop} \left\vert D \psi \right\vert \DIFF x \\
      &= h_T^{2k-\dtop} \int_{T} \norm{\Alt^k}{\mu_x}^2\DIFF x
      = h_T^{2k-\dtop}\norm{T}{\mu}^2.
    \end{aligned}
    \label{eq:Whitney.norm.P2}
  \end{equation}
  Both relations \eqref{eq:Whitney.norm} and \eqref{eq:Whitney.topo} follow by combining \eqref{eq:Whitney.norm.P1} with \eqref{eq:Whitney.norm.P2}.
\end{proof}

\section{Construction of a specific subspace $\cspace{k}(\Ih(f))$}\label{section:algorithm}

In this section, we build the specific subspace $\cspace{k}(\Ih(f)) \subset \cycle{k}(\Ih(f))$ appearing in \eqref{eq:relative.decomposition} by prescribing a suitable spanning set $\mc B_k$ for it.

\begin{algorithm}[H]
  \caption{Algorithm to construct a specific spanning set $\mc B_k$ on a given $d$-cell $f \in \Mh$.}
  \begin{algorithmic}[1]
    \Procedure{ConstructSpanningSet}{$k, f$}
    \State $\mc B_k \coloneqq \emptyset$
    \State $\mc F_k \coloneqq \emptyset$
    \State $\mc V_k \coloneqq \FM{k}(\Ih(\partial f))$ \Comment{ set of visited $k$-simplices}
    \For{$F \in \FM{k}(\Ih(f)) \setminus \FM{k}(\Ih(\partial f))$} \Comment{ loop over $k$-simplices not contained in $\partial f$}
    \If{$\exists z \in \cycle{k}(\Ih(f))$ such that $\supp(z) \subset \mc V_k \cup \{F\}$ and $\inner{F}{z} \neq 0$}
        \State let $z^*$ be the minimal-norm solution $z$ to the system:
               $ \left\{ \begin{array}{l}
                   \bd{k}z=0 \\
                   \inner{F}{z}=1 \\
                   \supp(z) \subset \mc V_k \cup \{F\}
                 \end{array} \right. $
        \State $\mc B_k \gets \mc B_k \cup \{z^*\}$
        \State $\mc F_k \gets \mc F_k \cup \{F\}$
    \Else
        \State $\mc V_k \gets \mc V_k \cup \{F\}$
    \EndIf
    \EndFor
    \State \textbf{return} $(\mc B_k, \mc F_k)$
    \EndProcedure
  \end{algorithmic}
  \label{acyclic_algo}
\end{algorithm}

Algorithm \ref{acyclic_algo} eventually terminates since the set $\Ih(f)$ is finite.
Let $\mc B_k$, $\mc F_k$, and $\mc V_k$ denote the sets obtained upon termination.
Observe that every $k$-simplex $F \in \FM{k}(\Ih(f)) \setminus \FM{k}(\Ih(\partial f))$ is added either to $\mc F_k$ or to $\mc V_k$ by the condition at line 6.
Together with the initialization of $\mc V_k$ at line 4, this yields the partition
\begin{equation}
\label{eq:partition}
\FM{k}(\Ih(f)) = \mc F_k \sqcup \mc V_k.
\end{equation}
For any subset $\mc X \subset \FM{k}(\Ih(f))$, define
\[
\cycle{k}(\mc X) \coloneqq \{\, z \in \cycle{k}(\Ih(f)) \;\st\; \supp(z) \subset \mc X \,\}.
\]
Notice that the ``if'' statement at line~6 of Algorithm~\ref{acyclic_algo} precisely checks whether there exists $z \in \cycle{k}\bigl(\mc V_k \cup \{F\}\bigr)$ such that $\inner{F}{z} \neq 0$.

We now prove the correctness of Algorithm \ref{acyclic_algo} via a sequence of elementary lemmas.
\begin{lemma}[Acyclicity property]
\label{lem:acyclic}
It holds $\cycle{k}(\mc V_k) = \cycle{k}(\Ih(\partial f))$.
\end{lemma}
\begin{proof}
The set $\mc V_k$ is initialized as $\FM{k}(\Ih(\partial f))$.
Then, at each iteration of the for loop in line 5, a $k$-simplex $F$ is added to $\mc V_k$ precisely when every $z \in \cycle{k}(\mc V_k \cup \{F\})$ satisfies $\inner{F}{z} = 0$.
This condition implies $\supp(z) \subset \mc V_k$, and consequently $\cycle{k}(\mc V_k \cup \{F\}) = \cycle{k}(\mc V_k)$.
This property holds by induction until termination.
\end{proof}

Let $\mc F_k = \{F_1, \dots, F_p\}$ be an enumeration of the set $\mc F_k$ based on the order in which its elements are added during the execution of Algorithm \ref{acyclic_algo}, and let $\mc B_k = \{z_1, \dots, z_p\}$ be a corresponding enumeration of the set $\mc B_k$ such that each $z_i \in \cycle{k}(\mc V_k \cup \{F_i\})$ satisfies $\inner{F_i}{z_i}=1$ for $i \in \{1, \dots, p\}$.
\begin{lemma}[Duality property]
  \label{lem:duality}
  The following duality property holds:
  \begin{align}
    \inner{F_i}{z_j} = \delta_{ij} \qquad \forall i,j \in \{1, \dots, p\}.
    \label{eq:duality}
  \end{align}
\end{lemma}
\begin{proof}
The case $i = j$ holds by choice of $z$ in Line 7 of the algorithm.
For $i \neq j$, we have $\supp(z_j) \subset \mc V_k \cup \{F_j\}$, and by \eqref{eq:partition}, $F_i \notin \mc V_k$.
Since we also have $F_i \neq F_j$, it follows that $F_i \notin \supp(z_j)$.
Thus, $\inner{F_i}{z_j} = 0$.
\end{proof}

\begin{lemma}[Complementary subspace basis property]
\label{lem:basis}
Define
\begin{equation*}
\cspace{k}(\Ih(f)) \coloneqq \mathrm{span}\,\mc B_k \subset \cycle{k}(\Ih(f)).
\end{equation*}
Then, $\mc B_k$ is a basis of $\cspace{k}(\Ih(f))$ and the following decomposition holds
\[
\cycle{k}(\Ih(f)) = \cycle{k}(\Ih(\partial f)) \oplus \cspace{k}(\Ih(f)).
\]
\end{lemma}

\begin{proof}
By Lemma \ref{lem:duality}, the set $\mc B_k$ is linearly independent and hence it forms a basis for $\cspace{k}(\Ih(f))$.

To prove the direct sum, we first establish that the intersection is trivial.
Let $z \in \cycle{k}(\Ih(\partial f)) \cap \cspace{k}(\Ih(f))$.
Since $z \in \cspace{k}(\Ih(f))$, $z = \sum_{j=1}^p c_j z_j$ for some coefficients $c_j$.
As $z \in \cycle{k}(\Ih(\partial f))$, its support is disjoint from $\mc F_k$ (which is made of simplices that are internal to $f$), and thus $\inner{F_i}{z} = 0$ for all $i \in \{1, \dots, p\}$.
By \eqref{eq:duality}, we infer that 
\[
0 = \inner{F_i}{z} = \inner{F_i}{\sum_{j=1}^p c_j z_j} = \sum_{j=1}^p c_j\, \delta_{ij} = c_i.
\]
Hence $z=0$, which proves that $\cycle{k}(\Ih(\partial f)) \cap \cspace{k}(\Ih(f)) = \{0\}$.

Second, let $z \in \cycle{k}(\Ih(f))$.
Define
\[
z' \coloneqq \sum_{j=1}^{p} \inner{F_j}{z}\, z_j \in \cspace{k}(\Ih(f))
\]
and set $z'' \coloneqq z - z'$.
A direct computation using \eqref{eq:duality} shows that, for any $i \in \{1, \dots, p\}$,
\[
\inner{F_i}{z''}
= \inner{F_i}{z} - \inner{F_i}{\sum_{j=1}^p \inner{F_j}{z}\,z_j}
= \inner{F_i}{z} - \sum_{j=1}^{p} \inner{F_j}{z} \,\delta_{ij}=\inner{F_i}{z}-\inner{F_i}{z}=0.
\]
This shows that $\supp(z'') \cap \mc F_k = \emptyset$ and thus, by \eqref{eq:partition}, that $\supp(z'') \subset \mc V_k$.
Therefore, by Lemma \ref{lem:acyclic}, $z'' \in \cycle{k}(\Ih(\partial f))$.
We have thus shown that any $z\in\cycle{k}(\Ih(f))$ can be written as $z = z'+ z''$ with $z' \in \cspace{k}(\Ih(f))$ and $z'' \in \cycle{k}(\Ih(\partial f))$. 
\end{proof}

Finally, we address the problem of bounding the coefficients of solutions to linear systems involving the boundary operators $\bd{k}: \chain{k}(\Ih(f)) \to \chain{k-1}(\Ih(f))$ for $k \in \{1, \dots, d\}$.
This will ensure that the selected elements of $\mc B_k$, and their pre-image by the boundary operator, have uniformly bounded $2$-norms (see \eqref{eq:bound.zi} and \eqref{eq:bound.wi}).

Let $\mat B_{k}$ be the matrix representation of $\bd{k}$.
In the case $k=1$, the boundary matrix $\mat B_1$ coincides with the graph incidence matrix (between $1$-cells and $0$-cells) of the triangulation $\Ih(f)$.
This matrix is known to be \emph{Totally Unimodular} (TU), meaning every sub-determinant is $0$, $+1$, or $-1$; see, e.g., \cite[Chap.~19]{schrijver1998theory}.
This TU property guarantees that the cycles $z \in \mc B_1$ selected by Algorithm \ref{acyclic_algo} (which are solutions to a TU system) have coefficients exclusively in $\{-1, 0, 1\}$.

In contrast, this TU property of $\mat B_k$ fails for $k \geq 2$ \cite[Chap.~20-21]{schrijver1998theory}.
We must therefore rely on a more general principle to bound the solutions, which is the goal of our next lemma.
This principle is based on Cramer's rule, combined with the fact that the boundary matrices $\mat B_k$ still have integer entries (specifically, in $\{-1, 0, 1\}$).
This general result allows us to establish two key estimates:
\begin{itemize}
\item First, it provides a uniform bound on the coefficients of each $k$-cycle $z \in \mc B_k$. As specified at line 7 in Algorithm \ref{acyclic_algo}, each such $z$ is selected as the unique minimal-norm solution to the well-posed linear system defined at that step.
\item Second, for $k \in \{1, \dots, d-1\}$, it allows us to bound the coefficients of a $(k+1)$-chain $w$ given a bound on its boundary $z = \bd{k+1}w$.
\end{itemize}
We state and prove this second result explicitly, and the first result follows from a similar argument.

\begin{lemma}
\label{lem:coeff.estimate}
Let $k \in \{1, \dots, d-1\}$, and let $z=\left(z_F\right)_{F \in \FM{k}(\Ih(f))}\in\cycle{k}(\Ih(f))$ be $k$-cycle.
Then, there exists a $(k+1)$-chain $w=\left(w_{F'}\right)_{F' \in \FM{k+1}(\Ih(f))} \in\chain{k+1}(\Ih(f))$ such that $\bd{k+1}w=z$ and, for some constant $C$ depending only on $\card(\FM{k}(\Ih(f)))$,
\begin{equation}\label{eq:bound.w.z}
\sum_{F' \in \FM{k+1}(\Ih(f))} w_{F'}^2 \,\leq\, C \sum_{F \in \FM{k}(\Ih(f))} z_F^2.
\end{equation}
\end{lemma}
\begin{proof}
Let $\mat B_{k+1}$ be the matrix representation of $\bd{k+1}$.
We identify the chains $w$ and $z$ with their column vectors of coefficients, $w = (w_{F'})_{F' \in \FM{k+1}(\Ih(f))}$ and $z = (z_F)_{F \in \FM{k}(\Ih(f))}$, as defined in the statement of the lemma.
We are thus solving the system $\mat B_{k+1} w = z$.
Since $f$ has a trivial topology and $z$ is a $k$-cycle with $k\ge 1$, we know that this system has at least one solution (there is a $w\in\chain{k+1}(\Ih(f))$ such that $\bd{k+1}w=z$).

Let $r = \text{rank}(\mat B_{k+1})$.
We can find a particular solution $w^*$ by solving an $r \times r$ invertible subsystem $\mat M w^*=z^*$, where $\mat M$ is a submatrix of $\mat B_{k+1}$, and $w^*$, $z^*$ are sub-vectors of $w$, $z$.
A solution to $\mat B_{k+1} w=z$ is then obtained by putting all components of $w$, other than those in $w^*$, to zero.

By Cramer's rule, the $F$-th component of the solution is $(w^*)_F = \det(\mat M_F) / \det(\mat M)$, where $\mat M_F$ is the matrix obtained by replacing the $F$-th column of $\mat M$ by $z^*$.
The entries of $\mat M$ are in $\{-1, 0, 1\}$. Since $\mat M$ is invertible and integer-valued, its determinant is a non-zero integer, thus $|\det(\mat M)| \ge 1$.

One of the columns of $\mat M_F$ is $z^*$, and the others $r-1$ columns are vectors of size $r$ with coefficients in $\{-1,0,1\}$, that therefore have a $2$-norm $\le \sqrt{r}$. Since the determinant is continuous (with norm 1) for the $2$-norm of vectors, we infer that $|\det(\mat M_F)|\le r^{(r-1)/2}\norm{2}{z^*}$.
Hence, $|(w^*)_F|\le r^{(r-1)/2}\norm{2}{z^*}$ and thus $\norm{2}{w^*}\le r^{r/2}\norm{2}{z^*}$. The same bound is therefore satisfied by $w$ obtained from $w^*$ by padding with zeros.
Since $r\le \card(\FM{k}(\Ih(f)))$, this concludes the proof of \eqref{eq:bound.w.z}.
\end{proof}

%------------------------------------------------------------------------------%
% Bibliography
%------------------------------------------------------------------------------%

\printbibliography

\end{document}